\documentclass{amsart}
\usepackage{amsmath, amsfonts, amssymb,amsbsy,eucal,mathrsfs}
\usepackage[all]{xy}
\usepackage{pstricks}
\usepackage{pst-plot}
\usepackage{graphicx}
\usepackage[utf8]{inputenc}
\usepackage[T1]{fontenc}
\usepackage{tikz}
\usepackage{chngcntr}

\newtheorem{thm}{Theorem}[subsection]
\counterwithin*{thm}{section}
\newtheorem{thm1}{Theorem}
\newtheorem{lemma}[thm]{Lemma}
\newtheorem{prop}[thm]{Proposition}
\newtheorem{cor}[thm]{Corollary}

\theoremstyle{definition}
\newtheorem{remark}[thm]{Remark}

\newtheorem{example}[thm]{Example}
\newtheorem{defn}[thm]{Definition}

\renewenvironment{proof}{{\flushleft \it Proof.}}{\hfill $\square$ \vspace{2mm}}

\DeclareMathOperator{\M}{\overline{M}}
\DeclareMathOperator{\Gr}{Gr}
\DeclareMathOperator{\E}{Eff}

\DeclareMathOperator{\LG}{LG}
\DeclareMathOperator{\IG}{IG}

\DeclareMathOperator{\OG}{OG}

\DeclareMathOperator{\Ker}{Ker}

\DeclareMathOperator{\HH}{H}
\DeclareMathOperator{\QH}{QH}
\DeclareMathOperator{\QHl}{QH(X)_{loc}}
\newcommand{\QHlh}[1]{\QH^{#1}(X)_{loc}}
\DeclareMathOperator{\QM}{QM}
\DeclareMathOperator{\BQH}{BQH}

\DeclareMathOperator{\codim}{codim}

\renewcommand{\P}{{\mathbb P}}
\newcommand{\C}{{\mathbb C}}
\newcommand{\R}{{\mathbb R}}
\newcommand{\Z}{{\mathbb Z}}

\newcommand{\htt}{{\textrm{ht}}}

\newcommand{\cO}{{\mathcal O}}

\newcommand{\pt}{{{\rm pt}}}
\newcommand{\cpt}{[\textrm{pt}]}
\newcommand{\id}{\text{id}}

\newcommand{\Bl}{\text{B$\ell$}}

\newcommand{\ev}{\operatorname{ev}}

\newcommand{\pic}[2]{\includegraphics[scale=#1]{#2}}

\newcommand{\be}{{\bf e}}
\newcommand{\ignore}[1]{}

\newcommand{\starz}{{\!\ \star_{0}\!\ }}

\def\bsigma0{\bar \sigma^0}

\def\pic{{\rm Pic}}

\def \Rh {{{\widehat{R}}}}

\def \a {{\alpha}}
\def \be {{\beta}}

\def \im {{\rm Im}}

\def\s{\sigma}\def\OO{\mathbb O}\def\PP{\mathbb P}

\renewcommand{\O}{\mathbb{O}}

\newcommand{\p}{\mathbb{P}}

\renewcommand{\a}{{\alpha}}

\newcommand{\rk}{{\rm rk}}

\newcommand{\scal}[1]{\langle #1 \rangle}

\newcommand{\g}{\mathfrak g}

\newcommand{\End}{{\rm End}}

\newcommand{\tr}{{\rm Tr}}

\newenvironment{proof-of-prop-}{{\flushleft \it Proof of Proposition \ref{prop-irr=>rc}.}}{\hfill $\square$ \vspace{2mm}}

\begin{document}

\title{Semisimple quantum cohomology\\
of some Fano varieties}

\date{25.04.2014}

\author{Nicolas Perrin}
\address{Math. 
Institut, Heinrich-Heine-Universit{\"a}t,
D-40204 D{\"u}sseldorf, Germany}
\email{perrin@math.uni-duesseldorf.de}


\thanks{}

\begin{abstract}
We give sufficient conditions for the semisimplicity of quantum
cohomology of Fano varieties of Picard rank 1. We apply these
techniques to prove new semisimplicity results for some Fano varieties of
Picard rank 1 and large index. We also give examples of Fano varieties
having a non semisimple small quantum cohomology but a semisimple big
quantum cohomology. 
\end{abstract}

\maketitle

\markboth{N.~PERRIN}
{SEMISIMPLE QUANTUM COHOMOLOGY}

\section*{Introduction}

Since Dubrovin's conjecture \cite{dubrovin}, the question whether the
big quantum cohomology ring of a variety is semisimple is important
and has been discussed in many articles
\cite{manin,bayer-m,HMT,irr,ciolli,semisimple}. In particular
necessary conditions for semisimplicity are given by Hertling, Manin
and Teleman in \cite{HMT}. For a few varieties $X$, for example some
Fano threefolds \cite{ciolli}, toric varieties \cite{irr} or some
homogeneous spaces \cite{semisimple}, it was proved that the small
quantum cohomology ring $\QH(X)$ (see Subsection
\ref{subsection-small}) is semisimple. However for some homogeneous
spaces, it was proved in \cite{semisimple} and \cite{adjoint} that the
small quantum cohomology does not need to be semisimple.

In this paper we give sufficient conditions for the small quantum cohomology ring $\QH(X)$ of a smooth complex Fano variety $X$ to be semisimple. For $X$ such a Fano variety and $\starz$ the product in $\QH(X)$ (see Subsection \ref{subsection-small}), we define
$$Q_X(a,b) = \textrm{ sum of the coefficients of $q^k$ for some $k$ in the product $a \starz b$}.$$
This is a quadratic form defined on $\QH(X)$ and if $R(X)$ is the radical of $\QH(X)$ we prove (see Theorem \ref{thm-alg-qh})

\begin{thm1}
\label{main1}
Let $X$ be Fano with Picard number $1$ and $Q_X$ positive definite. Let $h$ the class of a generator of the Picard group.

1. Then $R(X) \subset \{a \in \QH(X) \ | \ h^k \starz a = 0 \textrm{ for some $k$} \}$.

2. If $h$ is invertible in $\QH(X)$, then $\QH(X)$ is semisimple.
\end{thm1}

In the second part of this paper we give examples of varieties whose quadratic form $Q_X$ is positive definite. In particular we obtain (See Theorem \ref{thm-def-pos} and Proposition \ref{prop-def-pos})

\begin{thm1}
\label{main2}
Let $X$ be a variety in the following list (See Subsection \ref{subsect-qh-comin})

\medskip

\centerline{\begin{tabular}{lllllll}
$\p^n$ & & $Q_n$ & & $\OG(5,10)$ & & $E_6/P_6$ \\
$\Gr(2,n)$ & & $\LG(3,6)$ & & $\OG(6,12)$ & & $E_7/P_7$ \\
  \end{tabular}}

\medskip

\noindent
or an adjoint variety (See Subsection \ref{subsection-adjoint}) and let $Y$ be a general linear section of codimension $k$ of $X$ with $2 c_1(Y) > \dim Y$. Then $Q_Y$ is positive definite.
\end{thm1}

In particular alsmost all Fano varieties of coindex 3 occur in the
above list (see Example \ref{sub-exqm}). In the third part we give a closer look at the product $h \starz -$ with the generator of the Picard group and obtain the following semisimplicity result (See Theorem \ref{thm-semisimple-Y}).

\begin{thm1}
\label{main3}
Let $Y$ be a general hyperplane section with $2 c_1(Y) > \dim Y$ of a homogeneous space in the following list
$$\begin{array}{lllllll}
\p^n & & Q_n & & \LG(3,6) & & F_4/P_1 \\
\Gr(2,2n+1) & & \OG(5,10) & & \OG(2,2n+1) & & G_2/P_1.\\
\end{array}$$
Then $\QH(Y)$ is semisimple.
\end{thm1}

In particular, this Theorem recovers in a uniform way semisimplicity results proved \cite{adjoint} and \cite{pech} and provides new semisimplicity results. 

In the last section we consider two cases where the small quantum cohomology is not semisimple and prove using Theorem \ref{main1} that the big quantum cohomology ring (denoted $\BQH(X)$, see Section \ref{sec-bqh}) is semisimple (see Theorem \ref{thm-bqh1} and Theorem \ref{thm-bqh2}).

\begin{thm1}
\label{main4}
Let $X = \IG(2,2n)$ or $X = F_4/P_4$. Then $\QH(X)$ is not semisimple but $\BQH(X)$ is semisimple.
\end{thm1}

This result was the starting point of this work which came from
discussions with A.~Mellit and M.~Smirnov. They obtain in \cite{SM}
together with S.~Galkin an independent proof of the semisimplicity of
$\BQH(Y)$ for $Y = \IG(2,6)$.

\medskip

Let us say few words on Dubrovin's conjecture. Recall that the first part of this conjecture states that for $X$ smooth projective, the semisimplicity of $\BQH(X)$ is equivalent to the existence of a full exceptional collection in $D^b(X)$, where $D^b(X)$ denotes the bounded derived category of coherent sheaves on $X$. 

For all homogeneous spaces $X$ appearing in Theorem \ref{main3} and
Theorem \ref{main4} except $F_4/P_1$ and $F_4/P_4$, it is known (see
\cite[Section 6]{kuz1} and \cite{kuz2}) that their derived category
admits a full exceptional collection proving Dubrovin's conjecture in
these cases. 

The same is true for a hyperplane section $Y = \IG(2,2n+1)$ of $X = \Gr(2,2n+1)$ recovering results of \cite{pech}.

Furthermore, for $X = \Gr(2,5)$, $X = \OG(5,10)$ or $X = \LG(3,6)$ the hyperplane sections $Y$ of $X$ with $2 c_1(Y) > \dim Y$ also have an exceptional collection (see \cite[Section 6]{kuz1}) proving Dubrovin's conjecture in these cases.

\medskip

\textbf{Acknowledgement.} I thank A.~Mellit and M.~Smirnov for enlightening discussions and email exchanges.
I also thank the organisers of the conference {\em Quantum cohomology and quantum K-theory} held in Paris in January 2014 where this work started. Finally I thank P.-E. Chaput for the program \cite{programme} which was of great use in many computations.

\setcounter{tocdepth}{1}
\tableofcontents

\section{Big quantum cohomology}
\label{sec-bqh}

In this section we recall few facts and fix notation for the quantum cohomology of a complex smooth projective variety $X$. We write $\HH(X)$ for $H^*(X,\R)$.

\subsection{Reminders} Let $X$ be a smooth projective variety, let $\E(X)$ be the cone of effective curves and let $\be \in \E(X)$. Denote by $\M_{0,n}(X,\be)$  the Kontsevich moduli space of genus $0$ stable maps to $X$ of degree $\be$ with $n$ marked points. This is a proper scheme and there are evaluation morphisms $\ev_i : \M_{0,n}(X,\be) \to X$ defined by evalutating the map at the $n$-th marked points. For $(\gamma_i)_{i \in [1,n]}$ cohomology classes on $X$, one defines the Gromov-Witten invariants as follows:
$$I_{0,n,\beta}({\gamma_1,\cdots,\gamma_n}) = \int_{[\M_{0,n}(X,\be)]^{\textrm{vir}}} 
\ev_1^*\gamma_1 \cup \cdots \cup \ev_n^* \gamma_n$$
where $[\M_{0,n}(X,\be)]^{\textrm{vir}}$ is the virtual fundamental class as defined in \cite{beh-fant}.
When $\gamma_i = \gamma$ for all $i \in [1,n]$, we write $I_{0,n,\be}({\gamma_1,\cdots,\gamma_n}) = I_{0,n,\be}({\gamma^n})$.

Let $N + 1 = \rk(\HH(X))$ and $r = \rk(\pic(X))$. Let $(e_i)_{i \in [0,N]}$ be a basis of $\HH(X)$ such that $e_0 = 1$ is the fundamental class and $(e_1,\cdots,e_r)$ is a basis of $H^2(X,\R)$. For $\gamma \in \HH(X)$, write $\gamma = \sum_i x_i e_i$ and define the \textbf{Gromov-Witten potential} by
$$\Phi(\gamma) = \sum_{n \geq 0} \sum_{\be \in \E(X)} \frac{1}{n!} I_{0,n,\be}({\gamma^n}).$$
This is an element in $R = \R[[(x_i)_{i \in [0 , N]}]]$. If $(\ ,\ )$ denotes the Poincar\'e pairing, then the \textbf{quantum product $\star$} is defined as follows 
$$(e_i \star e_j , e_k) = \frac{\partial^3 \Phi}{\partial x_i \partial x_j \partial x_k}(\gamma)$$
and extended by bilinearity to any other classes. This actually defines a family of products parametrised by $\HH(X)$. The main result of the theory states that these products are associative (see \cite{behrend,beh-man,man-kon}). 
We write $\BQH(X)$ for the algebra $(\HH(X) \otimes_\R R , \star)$.

\subsection{Virtual fundamental class}

For our computations of Gromov-Witten invariants we shall use the following general result on the virtual fundamental class for smooth projective varieties (this was proved in \cite{ruan-tian} according to \cite[Point (1.4)]{beauville}, we refer to \cite[Proposition 2]{lee-k-theo} for an algebraic proof).

\begin{prop}
\label{prop-class-virt}
Let $X$ be a smooth projective complex algebraic variety such that $\bar M_{g,n}(X,\beta)$ has the expected dimension, then the virtual class is the fundamental class.
\end{prop}

\subsection{Divisor axiom}
One very useful property of Gromov-Witten invariants is that for degree $2$ cohomology classes, they can be easily computed. Indeed we have the following \textbf{divisor axiom}. For $\gamma_1 \in H^2(X,\R)$, we have
$$I_{0,n,\be}({\gamma_1,\cdots,\gamma_n}) = \left( \int_\be \gamma_1 \right) I_{0,n-1,\be}({\gamma_2,\cdots,\gamma_n}).$$
This gives a simplification of the potential (modulo terms of degree lower than 2):
$$\Phi(\gamma) = \gamma \cup \gamma \cup \gamma + \sum_{n \geq 0} \sum_{\be \in \E(X)} \frac{1}{n!} I_{0,n,\be}({\bar \gamma^n}) q^\be$$
where $\bar \gamma$ is the projection of $\gamma$ on the span of $(e_i)_{i \in [r+1,N]}$ and $q^\be = q_1^{d_1} \cdots q_r^{d_r}$ with $d_i = \int_\be e_i$ and $q_i = e^{x_i}$ for $i \in [1,r]$. Writing $\bar \gamma = \sum_{i = r + 1}^N x_i e_i$ we get
$$\Phi(\gamma) = \gamma \cup \gamma \cup \gamma + \sum_{n \geq 0} \sum_{\be \in \E(X)} \sum_{n_{r + 1} + \cdots + n_N = n}\frac{x_{r + 1}^{n_{ r + 1}} \cdots x_N^{n_N}}{n_{r + 1}! \cdots n_N!} I_{0,n,\be}({e_{r + 1}^{n_{r + 1}} , \cdots , e_N^{n_N}}) q^\be.$$

\subsection{Small quantum product}
\label{subsection-small}
A very classical special product called the small quantum product and denoted by $\star_0$ in this paper is obtained as follows
$$(e_i \star_0 e_j , e_k) = \frac{\partial^3 \Phi}{\partial x_i \partial x_j \partial x_k}(\gamma)\vert_{\bar \gamma = 0}.$$
This is also a family of associative products parametrised by $H^2(X,\R)$. This product is easier to compute and only involves $3$-points Gromov-Witten invariants \emph{i.e.} Gromov-Witten invariants with $n = 3$.  Set $R_0 = \R[[(q_i)_{i \in [1 , r]}]]$. We write $\QH(X)$ for the algebra $(\HH(X) \otimes_\R R_0 , \star_0)$. 
Note that $R_0 = R/((x_i)_{i \in [r + 1,N]})$. 

\subsection{Deformation in the $\tau$-direction}
\label{sub-def}

Let $\tau \in \HH(X)$ and choose the basis $(e_i)_{i \in [0,N]}$ so that $e_{r + 1} = \tau$. For $\gamma \in \HH(X)$, write $\hat \gamma$ for its projection in the span of $(e_i)_{i \in [r + 2,N]}$. We define the product $\star_\tau$ as follows:
$$(e_i \star_\tau e_j , e_k) = \frac{\partial^3 \Phi}{\partial x_i \partial x_j \partial x_k}(\gamma)\vert_{\hat \gamma = 0}.$$
This is also a family of associative products parametrised by $H^2(X,\R) \oplus \R \tau$. Set $R_\tau = \R[[(q_i)_{i \in [1 , r]},x_{r + 1}]]$. We write $\BQH_\tau(X)$ for the algebra $(\HH(X) \otimes_\R R_\tau , \star_\tau)$. Note that $R_\tau = R/((x_i)_{i \in [r + 2,N]})$. 
Let us describe a general product in this algebra (we set $t = x_{r + 1}$):
$$(e_i \star_\tau e_j , e_k) = (e_i \star_0 e_j , e_k) + t \sum_{\be \in \E(X)} I_{0,4,\be}({e_i,e_j,e_k,e_{r + 1}}) q^\be + O(t^2).$$
In particular when $e_i$ is the class of a divisor this product takes a simple form. Denote by $\Psi_i$ the endomorphism of $\R[[(q_i)_{i \in [1,r]}]]$ defined by
$$\Psi_i \left( \sum_{\be \in \E(X)} z_\be q^\be \right) = \sum_{\be
  \in \E(X)} d_i z_\be q^\be$$
with $d_i = \int_\be e_i$ and extend $\Psi_i$ by linearity on 
$\QH(X)$ via its actions on the scalars. We get the formula
$e_i \star_\tau e_j = e_i \star_0 e_j + t \Psi_i(e_{r+1}
\star_0 e_k) + O(t^2).$

\section{localisation of the radical}
\label{section-rad}

In this section, we prove that the existence of positive definite hermitian or real forms imply semisimplicity or regularity results on finite dimensional commutative algebras. Let $A$ be a finite dimensional commutative $\C$-algebra with $1$. We write $R(A)$ for the radical of $A$. For $a \in A$, we write $E_a \in \End_A(A)$ for the endomorphism obtained by multiplication with $a$.

\subsection{Semisimplicity and inner product}

We first relate the semisimplicity of $A$ with the existence of an inner product.

\begin{prop}
\label{prop-invol}
The algebra $A$ is semisimple if and only if there exists an algebra involution $a \mapsto \bar a$ and a linear form $\varphi : A \to \C$ with $\varphi(1) = 1$ such that the bilinear form defined by $Q(a,b) = \varphi(a \bar b)$ is an inner product.
\end{prop}

\begin{proof}
Let $n = \dim A$. Assume that $A$ is semisimple, then for $a \in A$, the endomorphism $E_a \in \End_A(A)$ is semisimple. Since $A$ is commutative, the endomorphisms $(E_a)_{a \in A}$ are simultaneously diagonalisable in a basis $(e_i)_{i \in [1,n]}$. These elements are orthogonal idempotents whose sum is 1. Define the linear form $\varphi$ by $\varphi(a) = n^{-1} \tr(E_a)$ and the involution $a \mapsto \bar a$ as the unique antilinear map with $\bar e_i = e_i$. This defines the desired inner product.

Conversely, assume that such an algebra involution and linear form exist. Then the endomorphisms $E_a$ are normal for $Q$: we have $Q(E_a(b),c) = Q(ab,c) = \varphi(ab \bar c) = Q(b,\bar a c) = Q(b,E_{\bar a}(c))$. The adjoint of $E_a$ is $E_{\bar a}$ and they commute. In particular the endomorphisms $(E_a)_{a \in A}$ are simultaneously diagonalisable in a basis $(e_i)_{i \in [1,n]}$ and these elements are orthogonal idempotents whose sum is 1.
\end{proof}

This result was first motivated by the following example.

\begin{example}
\label{exam-comin}
We refer to Subsection \ref{subsect-qh-comin} for results on quantum cohomology of cominuscule homogeneous spaces. Let $X$ be a cominuscule homogeneous space and let $A(X) = \QH(X)_{q=1}$ be its small quantum cohomology with product $\starz$ and with quantum parameter equal to 1. Let $\pt$ the the cohomology class of a point and let $\s_u$ be a Schubert class. Consider $\textrm{PD}(\pt \starz \s_u)$ where $\textrm{PD}$ stands for Poincar\'e duality. It was proved in \cite{strange} and \cite{semisimple}, that this class is a Schubert class $\s_{\bar u}$ and that $\s_u \mapsto \s_{\bar u}$ defines an algebra involution. Define furthermore a linear form $\varphi$ by $\varphi(\s_u) = \delta_{u,1}$ on the Schubert basis (recall that $1 = \s_1$). One easily checks that $Q(\s_u, \s_{\bar v}) = \delta_{u,v}$ proving that $Q$ is an inner product. We recover this way a result of \cite{semisimple} relating the semisimplicity with the existence of an algebra involution. 
\end{example}

\subsection{Radical and positive definite forms}

One of the major problems for applying the above result is that the algebra involution is not \emph{a priori} given and is usually hard to produce (for an example see \cite[Remark 6.6]{adjoint}). In this section we furthermore assume that $A$ is a $\R$-algebra which is $\Z/c_1\Z$-graded and denote by $A_k$ the graded piece of degree $k$.

\begin{defn}
Let $E \in \End(A)$. We set $A_\lambda(E) = \{ a \in A \ | \ (E - \lambda \id_A)^n(a) = 0 \textrm{ for $n$ large} \}$ and $m_E(\lambda) = \dim A_\lambda(E)$.
\end{defn}

\begin{lemma}
\label{prop-alg}
Let $E \in \End_A(A)$ with real eigenvalues. We have 
$$R(A) \subset \bigoplus_{m_E(\lambda) > 1} A_\lambda(E).$$
In particular if $E$ is semisimple regular, we have $R(A) = 0$.
\end{lemma}

\begin{remark}
Note that we assume here that $E$ is $A$-linear. This is for example the case for $E = E_a$ with $a \in A$.
\end{remark}

\begin{proof}
We have a decomposition 
$A = \bigoplus_\lambda A_\lambda(E).$
Furthermore by $A$-linearity, we have $A_\lambda(E) \cdot A_\mu(E) \subset A_\lambda(E) \cap A_\mu(E) = 0$ for $\lambda \neq \mu$. Let $a \in R(A)$ and write $a = \sum_\lambda a_\lambda$ with $a_\lambda \in A_\lambda(E)$. Then $a_\lambda$ is nilpotent. Write $1 = \sum_\lambda 1_\lambda$ with $1_\lambda \in A_\lambda(E)$. Then $1_\lambda$ is an idempotent and for $m_E(\lambda) = 1$ we have $a_\lambda = x_\lambda 1_\lambda$ with $x_\lambda \in \C$. Since $a_\lambda$ is nilpotent we have $x_\lambda = 0$ and the result follows.
\end{proof}

\begin{prop}
\label{main-prop}
Assume that there exists $\varphi_0 : A_0 \to \C$ a linear form with $\varphi_0(1) = 1$ such that the bilinear form defined by $Q_0(a,b) = \varphi(a b)$ for $a,b \in A_0$ is positive definite.

1. Then $R(A) \subset A_0(E_a)$ for all $a \in A_1$.

2. If there exists $a \in A_1$ with $\Ker E_a = 0$, then $A$ is semisimple.

3. If there exists $a \in A_1$ with $A_0 = A_0 \cap \Ker E_a \oplus A_0 \cap \im E_a$, then $R(A) \subset \Ker E_a$, $A = \im E_a \oplus \Ker E_a$ and the subalgebra generated by $a$ is semisimple.
\end{prop}

\begin{proof}
1. Let $n = \dim A_0$. As in the proof of Proposition \ref{prop-invol}, the endomorphisms $E_a\vert_{A_0}$ are self-adjoint for $Q_0$. In particular they have real eigenvalues and there exists a basis $(e_i)_{i \in [1,n]}$ of orthogonal idempotents whose sum is 1 in $A_0$. Let $a \in A_1$. Then $a^{c_1} \in A_0$ and $E_{a^{c_1}}\vert_{A_0}$ is semisimple. Furthermore there exists $b \in A_0^\times$ such that $E_{(ab)^{c_1}}\vert_{A_0}$ is semisimple and has eigenvalues with multiplicity 1 except maybe for $0$. The minimal polynomial $\mu$ of $E_{(ab)^{c_1}}\vert_{A_0}$ therefore has simple roots. Since $1 \in A_0$ this implies that $\mu((ab)^{c_1}) = 0$ and the minimal polynomial of $E_{(ab)^{c_1}}$ divides $\mu$ and thus has simple roots. In particular $E_{(ab)^{c_1}}$ is semisimple and has the same eigenvalues as $E_{(ab)^{c_1}}\vert_{A_0}$. 

Let $v$ be an eigenvector of $E_{ab}$ with eigenvalue $\lambda \neq 0$. Write $v = \sum_kv_k$ with $v_k \in A_k$. Then $v_k$ is an eigenvector of $E_{(ab)^{c_1}}$ for the eigenvalue $\lambda^{c_1}$. In particular $v_0$ is the unique (up to scalar) eigenvector of $E_{(ab)^{c_1}}$ with eigenvalue $\lambda^{c_1}$. From $E_{ab}(v) = \lambda v$ we deduce $E_{(ab)^{c_1-k}}(v_k) = \lambda^{c_1-k} v_0$. Applying $E_{(ab)^k}$ we get $\lambda^{c_1} v_k = E_{(ab)^{c_1}}(v_k) = \lambda^{c_1-k} E_{(ab)^k}(v_0)$.
Finally we have
$$v = \sum_{k = 0}^{c_1 -1} \lambda^{-k} E_{(ab)^{k}}(v_0).$$
Note furthermore that for $\lambda \neq 0$ we have $A_\lambda(E_{ab}) = \Ker(E_{ab} - \lambda \id)$ since $E_{(ab)^{c_1}}$ is semisimple. Therefore $m_{E_{ab}}(\lambda) = 1$ for $\lambda \neq 0$. By the previous Lemma, this implies $R(A) \subset A_{0}(E_{ab}) = A_0(E_{a})$ (the last equality holds since $b$ was chosen invertible).

2. If $a$ is invertible, then $A_0(E_a) = 0$ and the result follows from 1.

3. Let $a \in A_1$. Let $\mu$ be the minimal polynomial of $E_{a^{c_1}}$. Since $E_{a^{c_1}}$ is semisimple, we have $\mu(X) = X P(X)$ with $P$ such that $P$ has only simple roots and $P(0) \neq 0$. The minimal polynomial of $E_a$ divides $X^{c_1} P(X^{c_1})$. Let $b = P(a^{c_1}) \in A_0$ and write $b = ac + d$ with $c \in A$ and $d \in A_0 \cap \Ker E_a$. 

\begin{lemma}
We have $a^ib = 0$ for $i \geq 1$
\end{lemma}

\begin{proof}
Per descending induction on $i$. For $i \geq {c_1}$ the result follows from $0 = \mu(a^{c_1}) = a^{c_1} b$. Assume $a^ib = 0$ for $i > 1$. Then $a^iP(a^{c_1}) = 0$ and the minimal polynomial of $E_a$ divides $X^iP(X)$. In particular $\Ker E_{a^{i+1}} = \Ker E_{a^i}$. We have $0 = a^{i}b = a^{i+1} c +a^i d = a^{i+1} c$. Thus $c \in \Ker E_{a^{i+1}} = \Ker E_{a^i}$. This implies $a^{i-1} b = a^i c + a^{i-1}d = a^i c = 0$. 
\end{proof}

In particular the minimal polynomial of $E_a$ divides $X P(X^{c_1})$ and therefore has only simple roots (over $\C$). This implies that the subalgebra generated by $a$ is semisimple. It also implies that $E_a$ is semisimple and $A = \im E_a \oplus \Ker E_a$. This finally implies $A_0(E_{a}) = \Ker E_a$ and the result follows from 1. 
\end{proof}

\begin{remark}
The above proposition works actually for any $a \in A_k$ such that $\gcd(k,c_1) = 1$
\end{remark}

\subsection{Application to quantum cohomology of Fano varieties}
\label{sect-qh}

Let $X$ be a smooth complex projective Fano variety of Picard rank $1$ and let $H$ be a divisor such that $\cO_X(H)$ is an ample generator of the Picard group. The index $c_1(X)$ of $X$ is defined via $-K_X = c_1(X) H$. We write $\HH(X) = H^*(X,\R)$, $h \in H^2(X,\Z)$ for the cohomology class of $H$ and $\cpt \in H^{2 \dim X}(X,\Z)$ for the cohomology class of a point in $X$.

We denote by $\QH(X)$ the small quantum cohomology ring obtained using only $3$-points Gromov-Witten invariants. We write $\starz$ for the product in $\QH(X)$. Recall that as $\R$-vector space we have $\QH(X) = \HH(X) \otimes_\R \R[q]$. The ring $\QH(X)$ is graded with $\deg(q) = 2 c_1(X)$. We write $\QH^{k}(X)$ for the degree $k$ graded piece. We let $A = \QH(X)_{q=1}$ be the algebra obtained from $\QH(X)$ by quotienting with the ideal $(q-1)$. This algebra is $\Z/2c_1(X)\Z$-graded and we write $A_k$ for its degree $k$ graded piece.

In particular we have a decomposition
$$A_0 = \bigoplus_{k \geq 0} H^{2 k c_1(X)}(X,\R).$$
Denote by $\varphi_0 : A_0 \to \R$ the projection on the first factor and let $Q_X: A_0 \times A_0 \to \R$ be the quadratic form defined by $Q_X(\alpha,\beta) = \varphi_0(\alpha \starz \beta)$.

\begin{thm}
\label{thm-alg-qh}
Let $X$ be Fano with Picard number $1$ and $Q_X$ positive definite.

1. Then $R(A) \subset A_0(E_{h})$.

2. If $\Ker E_h = 0$, then $A$ is semisimple.

3. If $A_0 = A_0 \cap \Ker E_h \oplus A_0 \cap \im E_h$, then $R(A) \subset \Ker E_h$, $A = \im E_h \oplus \Ker E_h$ and the subalgebra generated by $h$ is semisimple.
\end{thm}

\begin{proof}
Follows from Proposition \ref{main-prop}.
\end{proof}

\section{Varieties with $Q_X$ positive definite}
\label{section-QX}

In this section we give examples of Fano varieties $X$ with $Q_X$ positive definite. These varieties are obtained as complete intersections of homogeneous spaces. The commun feature of our homogeneous spaces is the fact that the variety of conics passing through a point has a positive definite intersection form on its middle cohomology. Note that this positivity property implies strong topological properties (see for example \cite{small-codim} and \cite{small-cras}).

\subsection{Fano of large index}

We start with general remarks on $Q_X$ for Fano varieties of large index. More precisely, we assume $2 c_1(X) > \dim X$. Note that with this assumption we have 
$$A_0 = H^0(X,\R) \oplus H^{2 c_1(X)}(X,\R).$$ 
Furthermore the fact that $1 \in H^0(X,\R)$ is a unit for the quantum cohomoloy implies $Q_X(1,1) = 1$ and $Q_X(1,\s) = 0$ for $\s \in H^{2 c_1(X)}(X,\R)$. To compute $Q_X$, we therefore only have to compute $\s \starz \s'$ for $\s , \s' \in H^{2 c_1(X)}(X,\R)$. By dimension arguments we have
$$\s \starz \s' = q \s'' + I_{2}(\s,\s',\pt) q^2$$
where $\s'' \in H^{2 c_1(X)}(X,\R)$ and $I_2(a,b,c) I_{0,3,2}(a,b,c)= $ is $3$-points Gromov-Witten invariant of degree $2$ in genus $0$ for the classes $a,b,c$. We obtain the following result.

\begin{lemma}
\label{lemm-c>dim}
Let $X$ be a Fano variety with $2 c_1(X) > \dim X$. Then $Q_X$ is positive definite if and only if $I_2(-,-,\pt)$ is positive definite on $H^{2 c_1(X)}(X,\R)$.
\end{lemma}

For computing $I_2(-,-,\pt)$ we shall prove using Proposition \ref{prop-class-virt} that the virtual fundamental class is the actual fundamental class.

\subsection{Complete intersections in projective spaces}
\label{subsection-complete-intersections}

We first consider complete intersections in $\p^{n+r}$ with large index. We reinterpret here results obtained by Beauville \cite{beauville}. Let $X$ be a smooth complete intersection of $r$ hypersurfaces of degree $(d_1,\cdots,d_r)$ in $\p^{n+r}$ with $n \geq 2$. Assume that $n \geq 2 \sum(d_i-1) -1$.

\begin{lemma}
\label{lemm-comp-int}
The form $Q_X$ is positive definite.
\end{lemma}

\begin{proof}
This easily follows from the results in \cite{beauville}. We have a basis $(1,h^{c_1(X)})$ of $A_0$ and $\varphi_0(h^{c_1(X)}) \neq 0$. Furthermore $h^{c_1(X)} \starz  h^{c_1(X)} = d_1^{d_1} \cdots d_r^{d_r} h^{c_1(X)}$ proving the result.
\end{proof}

\begin{remark}
Let $X$ be a complete intersection as above.

1. By Theorem \ref{thm-alg-qh}, we have $R(X) \subset A_0(E_h)$. Using the results of Beauville \cite{beauville} the primitive classes $H^{\dim X}_0(X,\R)$ in $H^{\dim X}(X,\R)$ (defined as the kernel of $E_h$) are nilpotent as well as the class $h^{c_1(X) + 1} - d_1^{d_1} \cdots d_r^{d_r} q h$. One can furthermore easily check that these classes generate the radical $R(X)$ of $\QH(X)$. So $\QH(X)$ is not semisimple in general.

2. Tian and Xu \cite{tian-xu} proved that the subalgebra generated by the hyperplane class in $\BQH(X)$ -- the big quantum cohomology -- is semisimple for any complete intersection as above.

3. We do not know in general whether $\BQH(X)$ is semisimple.
\begin{itemize}
\item By results of Hertling, Manin and Teleman \cite{HMT}, a variety has semisimple quantum cohomology only if its cohomology is even and pure of type $(p,p)$. By results of Deligne \cite{deligne} the only possible complete intersections are quadrics, the cubic surface and even-dimensional complete intersection of two quadrics. 
\item For the first three it is known that the (small) quantum cohomology is semisimple (see \cite{semisimple}, \cite{ciolli}).
\item For the last one, we do not know if $\BQH(X)$ is semisimple.
\end{itemize}
\end{remark}

\subsection{Cominuscule homogeneous spaces}
\label{subsect-qh-comin}

Let $G$ be a semisimple algebraic group. A parabolic subgroup $P$ is called \emph{cominuscule} if its unipotent radical $U_P$ is abelian. This group theoretic condition has many nice implications on the geometry of $X = G/P$ (\cite{seshadri}, \cite{small}, \cite{adv}, \cite{thomas-yong}). Table 1 gives a list of all cominuscule homogeneous spaces.

Recall that the vertices of the Dynkin diagram of $G$ are the simple roots. The marked vertex is the simple root of $G$ which is not a simple root of $P$. In the above table we denoted $\Gr(k,n)$ (resp. $\LG(n,2n),\OG(n,2n)$) the Gra\ss mann variety of $k$-subspaces in $\C^n$ (resp. isotropic $n$-subspaces for a symplectic or non degenerate quadratic form in $\C^{2n}$, for $\OG(n,2n)$ we only consider a connected component of the Gra\ss mann variety). We wrote $Q_n$ for a smooth $n$-dimensional quadric hypersurface and $\OO \PP^2 = E_6/P_6$ and $E_7/P_7$ are the Cayley plane and the Freudenthal variety.

$$\begin{array}{ccccccc}
Type & X & Diagram  & Dimension & 
c_1(X) & \dim(\Gamma_2) \\
\hline
 A_{n-1} & \Gr(k,n) &\setlength{\unitlength}{2.5mm}
\begin{picture}(10,3)(0,0)
\put(0,0){$\circ$}
\multiput(2,0)(2,0){5}{$\circ$}
\multiput(.73,.4)(2,0){5}{\line(1,0){1.34}}
\put(4,0){$\bullet$}
\end{picture} & k(n-k) & n & 4 \\
B_n & Q_{2n-1} &
\setlength{\unitlength}{2.5mm}
\begin{picture}(10,3)(0,0)
\put(0,0){$\circ$}
\multiput(2,0)(2,0){5}{$\circ$}
\multiput(.73,.4)(2,0){4}{\line(1,0){1.34}}
\multiput(8.73,.2)(0,.4){2}{\line(1,0){1.34}}
\put(0,0){$\bullet$}
\end{picture} & 2n-1 & 2n - 1 & 2n - 1 \\
C_n & \LG(n,2n) &
\setlength{\unitlength}{2.5mm}
\begin{picture}(10,3)(0,0)
\put(0,0){$\circ$}
\multiput(2,0)(2,0){5}{$\circ$}
\multiput(.73,.4)(2,0){4}{\line(1,0){1.34}}
\multiput(8.73,.2)(0,.4){2}{\line(1,0){1.34}}
\put(10,0){$\bullet$}
\end{picture} & \frac{n(n+1)}{2} & n+1 & 3 \\
D_n & Q_{2n-2} &
\setlength{\unitlength}{2.5mm}
\begin{picture}(10,3)(0,0)
\put(2,0){$\circ$}
\multiput(4,0)(2,0){4}{$\circ$}
\multiput(2.73,.4)(2,0){4}{\line(1,0){1.34}}
\put(10,0){$\bullet$}
\put(0,-1.1){$\circ$}
\put(0,1.2){$\circ$}
\put(.6,1.5){\line(5,-3){1.5}}
\put(.6,-.64){\line(5,3){1.5}}
\end{picture}
\vspace{.2cm}
 & 2n-2 & 2n-2 & 2n - 2 \\
D_n & \OG(n,2n) &
\setlength{\unitlength}{2.5mm}
\begin{picture}(10,3)(0,0)
\put(2,0){$\circ$}
\multiput(4,0)(2,0){4}{$\circ$}
\multiput(2.73,.4)(2,0){4}{\line(1,0){1.34}}
\put(0,1.2){$\bullet$}
\put(0,-1.1){$\circ$}
\put(.6,1.5){\line(5,-3){1.5}}
\put(.6,-.64){\line(5,3){1.5}}
\end{picture}
\vspace{.2cm}
 & \frac{n(n-1)}{2} & 2n-2 & 6 \\
E_6 & \OO\PP^2 & 
\setlength{\unitlength}{2.5mm}
\begin{picture}(10,3)(-1,-.5)
\put(0,0){$\circ$}
\multiput(2,0)(2,0){4}{$\circ$}
\multiput(.73,.4)(2,0){4}{\line(1,0){1.34}}
\put(0,0){$\bullet$}
\put(4,-2){$\circ$}
\put(4.42,-1.28){\line(0,1){1.36}}
\end{picture}
& 16 & 12 & 8 \\
E_7 & E_7/P_7 & 
\setlength{\unitlength}{2.5mm}
\begin{picture}(10,3)(0,0)
\put(0,0){$\circ$}
\multiput(2,0)(2,0){5}{$\circ$}
\multiput(.73,.4)(2,0){5}{\line(1,0){1.34}}
\put(10,0){$\bullet$}
\put(4,-2){$\circ$}
\put(4.42,-1.28){\line(0,1){1.36}}
\end{picture}
& 27 & 18 & 10
\end{array}
$$

\medskip
\medskip
\medskip

\centerline{Table 1. List of cominuscule homogeneous spaces.}

\subsubsection{Cominuscule varieties with $Q_X$ positive definite}

Let $X$ be cominuscule. The following results on $\QH(X)$ were proved in \cite{semisimple} and \cite{strange}. Let $\pt$ the the cohomology class of a point and let $\s_u$ be a Schubert class. Then 
$$\pt \starz \s_u = q^{d(u)} \textrm{PD}(\s_{\bar u})$$
where $d(u)$ is a non negative integer and $u \mapsto \bar u$ is an involution on Schubert classes. It was also proved that $\s_u \mapsto q^{-d(u)} \s_{\bar u}$ defines an algebra involution. As explained in Example \ref{exam-comin} this defines an inner product on $\QH(X)_{q = 1}$. Note that proving that $Q_X$ is positive definite is equivalent to proving that the classes of degree a multiple of $c_1(X)$ are fixed by the above involution. Since the involution is explicit, an easy check gives the following result.

\begin{prop}
Let $X$ be cominuscule. The form $Q_X$ is positive definite if and only if $X$ is one of the following varieties:

\medskip

\centerline{\begin{tabular}{lll}
$\p^n$ & & $\LG(n,2n)$ for $n \in [3,4]$ \\
$\Gr(2,n)$ for $n \geq 2$ & & $\OG(n,2n)$ for $n \in [1,6]$ \\
$Q_n$ for $n \geq 2$ & & $\O\P^2$ or $E_7/P_7$. \\
  \end{tabular}}

\medskip

\end{prop}

\subsubsection{Geometric proof} We now give a geometric proof of the above result when $2 c_1(X) > \dim X$ (this only excludes $\LG(4,8)$ of the list). According to Lemma \ref{lemm-c>dim}, we only have to understand $I_2(-,-,\pt)$ on $H^{2 c_1(X)}(X,\R)$. We recall a geometric construction for cominuscule homogeneous spaces.

For $x,y \in X$, let $d(x,y)$ be the minimal degree of a rational curve passing through $x$ and $y$ and $d_X(2) = \max\{d(x,y) \ | \ x,y \in X\}$. We denote by $\Gamma_d(x,y)$ the union of all degree $d(x,y)$ rational curves passing through $x$ and $y$. Note that $d_X(2) = 1 $ if and only if $X$ is a projective space so we may assume $d_X(2) \geq 2$. The following result was proved in \cite{cmp1}.

\begin{prop}
\label{prop-cmp}
Let $d \in [0,d_X(2)]$ and let $x,y \in X$ with $d(x,y) = d$. 

1. The variety $\Gamma_d(x,y)$ is a homogeneous Schubert variety in $X$.

2. Any degree $d$ curve is contained in a $G$-translate of $\Gamma_d(x,y)$.

3. A generic degree $d$ curve is contained in a unique $G$-translate of $\Gamma_d(x,y)$.

4. There passes a unique degree $d$ curve through three general points in $\Gamma_d(x,y)$.
\end{prop}

For $d$ and $x,y \in X$ as in the former proposition, denote by $Y_d(X)$ the variety of all $G$-translates of $\Gamma_d(x,y)$. Since $\Gamma_d(x,y)$ is a Schubert variety, its stabiliser is a parabolic subgroup $Q$ and $Y_d(X) = G/Q$. Without loss of generality, we may assume that $P \cap Q$ contains a Borel subgroup. Let $Z_d(X) = G/(P \cap Q)$ be the incidence variety. Write $M_d(X)$ for the moduli space of stable maps of genus $0$ and degree $d$ with $3$ marked points to $X$. Let $\Bl_d(X) = \{ (\Gamma_d,f) \in Y_d(X) \times M_d(X) \ | \ f \textrm{ factors through } \Gamma_d \}$ and $Z_d^{(3)}(X) = \{ (\Gamma_d,x_1,x_2,x_3) \in Y_d(X) \times X^3  \ | \ x_i \in \Gamma_d \textrm{ for all $i \in [1,3]$} \}$. We have a diagram
$$\xymatrix{\Bl_d(X) \ar[rr]^-\pi \ar[d]_-\phi & & M_d(X) \ar[d]^-{\ev_i} \\
Z_d^{(3)}(X) \ar[r]^-{e_i} & Z_d(X) \ar[r]^-p \ar[d]_-q & X \\
& Y_d(X), & \\}$$
where $\pi$ is the natural projection, $\phi$ maps $(\Gamma_d,f)$ to $(\Gamma_d,\ev_1(f),\ev_2(f),\ev_3(f))$, $e_i$ maps $(\Gamma_d,x_1,x_2,x_3)$ to $(\Gamma_d,x_i)$ and $p$ and $q$ are the natural projections. The above proposition implies that $\pi$ and $\phi$ are both birational. Note also that the third point of Proposition \ref{prop-cmp} implies that considering $\Gamma_d \in Y_d(X)$ as a smooth subvariety in $X$, we have $2 \dim(\Gamma_d) = d c_1(\Gamma_d)$ (see also \cite[Formula (5) on Page 73]{cmp1}). For $d = 2$ this implies $\dim \Gamma_2 = c_1(\Gamma_2)$ so $\Gamma_2$ is a smooth quadric hypersurface (see also Table 1 for its dimension).

By Proposition \ref{prop-class-virt} and since $M_d(X)$ has expected dimension, for $\s,\s',\s'' \in \HH(X)$, the Gromov-Witten invariant $I_d(\s,\s',\s'')$ is the push-forward to the point of the class $\ev_1^*\s \cup \ev_2^*\s' \cup \ev_3^*\s''$. Since $\pi$ and $\phi$ are birational, an easy diagram chasing gives the following formula (usually called \emph{quantum to classical principle}):
\begin{equation}
\label{q-to-c}
I_d(,\s,\s',\s'') = q_*p^* \s \cup q_*p^* \s' \cup q_*p^* \s''.
\end{equation}
The very first version of this result was proved in \cite{bkt} for (maximal isotropic) Gra\ss mann varietes and generalised in \cite{cmp1}. For the formal computation in the above setting (and even in equivariant $K$-theory), we refer to \cite[Lemma 3.5]{BM}.

\begin{prop}
\label{prop-geom-comin}
Let $X$ be one of the following varieties

\medskip

\centerline{\begin{tabular}{lll}
$\p^n$ & & $\LG(n,2n)$ for $n \in [1,4]$ \\
$\Gr(2,n)$ for $n \geq 2$ & & $\OG(n,2n)$ for $n \in [1,6]$ \\
$Q_n$ for $n \geq 2$ & & $\O\P^2$ or $E_7/P_7$. \\
  \end{tabular}}

\medskip

Then $I_2(-,-,\pt)$ is positive definite.
\end{prop}

\begin{proof}
We apply formula \eqref{q-to-c}. Let $\s'' = \pt$, then $q_*p^*\pt = j_*[F]$ where $F$ is any fiber of $p$ and $j : F \to Y_d(X)$ the inclusion. Projection formula gives
$$I_d(,\s,\s',\s'') = I_d(,\s,\s',\pt) = q_*p^* \s \cup q_*p^* \s' \cup q_*p^* \pt = j^*q_*p^* \s \cup j^*q_*p^* \s'.$$
The following table gives the list of the varieties $F$ (see \cite[Table on Page 71]{cmp1}).

$$\begin{array}{c|cccccc}
X & \Gr(2,n) & Q_n & \LG(n,2n) & \OG(n,2n) & E_6/P_6 & E_7/P_7 \\
Y_2(X) & \Gr(4,n) & \{\pt\} & \IG(n-2,2n) & \OG(n-4,2n) & E_6/P_1 & E_7/P_1 \\
F & \Gr(2,n-2) & \{\pt\} & \Gr(2,n) & \Gr(4,n) & Q_8 & E_6/P_1. \\
\end{array}$$

\medskip

\centerline{Table 2. Varieties $Y_2(X)$ and $F$.}

\medskip

Note that since $\pi$ and $\phi$ are birational, we have $\dim X + 2 c_1(X) = \dim Z_2^{(3)}(X) = \dim Z_2(X) + 2 \dim \Gamma_2 = \dim X + \dim F+ 2 \dim \Gamma_2$. In particular we get 
$$\dim F = 2(c_1(X) - \dim \Gamma_2).$$
Thus $\deg j^*q_*p^* \s = \deg q_*p^*\s = 2 c_1(X) - 2 \dim \Gamma_2 = \dim F$ for $\s \in H^{2 c_1(X)}(X,\Z)$. We get an induced map
$$j^*q_*p^* : H^{2 c_1(X)}(X,\Z) \to H^{\dim F}(F,\Z).$$
Now for the varieties in the list of the proposition one easily checks (using Schubert classes) that this map is an isomorphism. Furthermore the variety $F$ is of even dimension and has positive definite Poincar\'e pairing on $H^{\dim F}(F,\Z)$ (see \cite[Table on Page 572]{small-codim} for the fact that the Poincar\'e pairing is positive definite on the middle cohomology of $\Gr(2,n)$, $Q_8$ and $E_6/P_1$). This finishes the proof.
\end{proof}

\subsection{Linear sections of cominuscule homogeneous spaces}

In this subsection, we extend the result on cominuscule varieties to linear sections of cominuscule varietes. More precisely we prove

\begin{thm}
\label{thm-def-pos}
Let $X$ be a cominuscule variety with $Q_X$ positive definite and let $Y$ be a general linear section of codimension $k$ of $X$ with $2 c_1(Y) > \dim Y$. Then $Q_Y$ is positive definite.
\end{thm}

\begin{remark}
\label{rem-def-pos}
The possible values of $k$ in the above Theorem are as follows:

\medskip

\centerline{\begin{tabular}{lll}
For $\p^n$ we have $k \leq n$ & & For $\OG(5,10)$ we have $k \leq 5$ \\
For $\Gr(2,n)$ we have $k \leq 3$ & & For $\OG(6,12)$ we have $k \leq 4$ \\
For $Q_n$ we have $k \leq n$ & & For $\O\p^2$ we have $k \leq 7$ \\
For $\LG(3,6)$ we have $k \leq 1$ & & For $E_7/P_7$ we have $k \leq 8$. \\
  \end{tabular}}
\end{remark}

We will prove this result in two steps. First note that since $2 c_1(Y) > \dim Y$ it is enough to check (using Lemma \ref{lemm-c>dim}) that $I_2(-,-,\pt)$ is positive definite. Let $L_k$ be the linear subspace of codimension $k = \dim X - \dim Y$ cuting $Y$ out of $X$. Note that since $c_1(Y) = c_1(X) -k$ and $\dim Y = \dim X -k$, the condition $2 c_1(Y) > \dim Y$ translates into $k < 2 c_1(X) - \dim X$ and that this implies $k < \dim \Gamma_2$ (where $\Gamma_2$ is the fiber of $q: Z_2(X) \to Y_2(X)$).

\subsubsection{Moduli space $M_2(Y)$} We first prove the following result asserting that $M_2(Y)$ has the expected dimension.

\begin{prop}
We have $\dim M_2(Y) = 2 c_1(Y) + \dim Y$.
\end{prop}

\begin{proof}
Note that we have a map of stacks $M_2(Y) \to \mathfrak{M}_{0,3}$
where $\mathfrak{M}_{0,3}$ is the stack of prestable curves of genus
$0$ with $3$ marked points. The fibers of this map are the schemes of
morphisms of degree $2$ from a fixed prestable curve to $X$ (see for
example \cite{behrend}). In
particular the irreducible components of any of these fibers has
dimension at least the corresponding expected dimension. 

Set $\Bl_2(Y) = \pi^{-1}(M_2(Y))$. We first prove that map $\Bl_2(Y) \to Y_2$ is surjective. Indeed. let $\Gamma_2 \in Y_2$. Then $\Gamma_2$ is a quadric of dimension $\dim \Gamma_2 > k$. Its intersection with $L_k$ therefore contains a conic. In particular, there exists a stable map in $M_2(Y)$ factorising through $\Gamma_2$ proving the surjectivity. The fiber of the map $\Bl_2(Y) \to Y_2(X)$ over $\Gamma_2$ is therefore given by the genus zero stable maps of degree 2 and three marked points to $\Gamma_2 \cap L_k$. We thus need to understand this intersection more precisely. 
We shall consider $L_k$ as the intersection of $k$ hyperplanes $H_1,\cdots,H_k$.
\begin{itemize}
\item[1.] the intersection $\Gamma_2 \cap L_k$ is a smooth quadric of dimension $\dim \Gamma_2 -k$,
\item[2.] the intersection $\Gamma_2 \cap L_k$ is a quadric of dimension $\dim \Gamma_2 -k$ and rank $\ell < \dim \Gamma_2 - k + 1$,
\item[3.] we have $\dim( \Gamma_2 \cap L_k) > \dim \Gamma_2 -k$.
\end{itemize}
Since the dimension of the moduli space of genus zero stable maps of degree 2 to a quadric of given rank is well known, an easy check proves that the locus in $\Bl_2(Y)$ over points $\Gamma_2 \in Y_2(X)$ such that the intersection $\Gamma_2 \cap L_k$ is for each prestable curve in $\mathfrak{M}_{0,3}$ of dimension strictly less than the expected dimension. In particular irreducible components of $M_2(Y)$ come from irreducible components of $\Bl_2(Y)$ containing points mapping in $Y_2$ to an quadric $\Gamma_2$ such that $\Gamma_2 \cap L_k$ is a smooth quadric of dimension $\dim \Gamma_2 - k$. Since any stable map to a such intersection $\Gamma_2 \cap L_k$ is a limit of a stable map from an irreducible curve we can consider only irreducible curves. This implies that $M_2(Y)$ is irreducible of expected dimension $\dim Y + 2 c_1(Y)$.
\end{proof}

\subsubsection{Proof of Theorem \ref{thm-def-pos}} Since $M_2(Y)$ has expected dimension, the virtual class is the fundamental class by Proposition \ref{prop-class-virt}. Let $Z_2(Y) = p^{-1}(Y)$ and denote by $r$ and $s$ the projections $r:Z_2(Y) \to Y$ and $s:Z_2(Y) \to Y_2(X)$. Since $\pi$ and $\phi$ restricted to $\Bl_2(Y)$ are again birational, the same computation as in the cominuscule case gives the relation
$$I_2(\tau,\tau',\tau'')_Y = r_*s^*\tau \cup r_*s^*\tau' \cup r_*s^*\tau''$$
where $I_2(-,-,-)_Y$ denotes the Gromov-Witten invariants in degree $2$ in $Y$. But the diagram
$$\xymatrix{Z_2(X) \ar[r]^-p & X \\
Z_2(Y) \ar[r]^-r \ar@{^(->}[u]^-i & Y \ar@{^(->}[u]_-j}$$
is Cartesian with $p$ flat (it is a locally trivial fibration since $X = G/P$, see for example \cite[Proposition 2.3]{BCMP}). In particular we have $i_*r^* = p^*j_*$. We thus get $s_*r^*\tau = q_*i_*r^* \tau = q_*p^*j_*\tau$. We deduce
$I_2(\tau,\tau',\tau'')_Y 
= I_2(j_*\tau , j_*\tau' , j_*\tau'')_X.$
In particular the result follows since $I_2(-,-,\pt)_X$ is positive definite.

\begin{remark}
\label{rem-thm-def-pos}
Note that we proved more than Theorem \ref{thm-def-pos}. Indeed, for any cohomology classes $\tau,\tau',\tau'' \in \HH(Y)$ we have the equality
$$I_2(\tau,\tau',\tau'')_Y 
= I_2(j_*\tau , j_*\tau' , j_*\tau'')_X.$$
\end{remark}

\subsubsection{Examples}
\label{sub-exqm}
 Several linear sections satisfying Theorem \ref{thm-def-pos} are classical.

1. Hyperplane sections of the Gra{\ss}mann variety $\Gr(2,n)$. The Pl\"ucker embedding is given by the representation $\Lambda^2\C^n$. A general linear section corresponds to a general symplectic form on $\C^n$ and the hyperplane section is the subvariety of isotropic $2$-dimensional subspaces. 

For $n = 2p$ even, the variety $Y = \IG(2,2p)$ is homogeneous. 

For $n = 2p + 1$ odd, the variety $Y = \IG(2,2p+1)$ is not homogeneous. This variety has two orbits under its automorphism group and is known as the odd symplectic Gra{\ss}mann variety of lines. We refer to \cite{mihai,pech,pasquier,pp} for several geometric results on this variety.

2. Hyperplane sections of $\O\p^2 = E_6/P_6$ are homogeneous under the group $F_4$. Actually we have $Y = F_4/P_4$ which is the coadjoint variety of type $F_4$ (see \cite{land-man,adjoint}).

3. For $X = \Gr(2,5)$ and $k = 3$, then $Y =V_5$ is the \emph{del Pezzo} threefold of index $2$ and degree $5$.

4. Note that we obtain almost all Fano varieties $X$ of coindex 3 \emph{i.e.} with $c_1(X) = \dim X -2$: we obtain all Fano varieties $X$ of coindex 3 with genus $g \in [7,10]$ \emph{i.e.} missing the extremal values $g = 6$ and $g = 12$ (see \cite[Theorem 5.2.3]{fano}).

\subsection{Adjoint varieties}
\label{subsection-adjoint}

The last family of varietes with $Q_X$ positive definite are the adjoint varieties (see \cite{adjoint}). These are homogeneous spaces and can be defined, for $G$ a semisimple group, as the closed $G$-orbit in $\p\g$ where $\g$ is the Lie algebra of $G$ and $G$ acts on its Lie algbera by via the adjoint representation.
The list of adjoint varieties is given in Table 3. Note that the following equality holds: $\dim X = 2 c_1(X) - 1$ (except in type $C_n$).

The following result was proved in \cite[Proof of Proposition 6.5]{adjoint}

\begin{prop}
\label{prop-def-pos}
Let $X$ be an adjoint variety. Then $Q_X$ is positive definite.
\end{prop}

$$\begin{array}{ccccc}
Type & variety & diagram  & dimension & \hspace*{5mm}index \hspace*{5mm}\\
 A_{n} &\hspace{5mm} \textrm{Fl}(1,n\ ;n+1) \hspace{5mm} & \setlength{\unitlength}{2.5mm}
\begin{picture}(14,3)(-2,0)
\put(0,0){$\circ$}
\multiput(2,0)(2,0){5}{$\circ$}
\multiput(.73,.4)(2,0){5}{\line(1,0){1.34}}
\put(0,0){$\bullet$}
\put(10,0){$\bullet$}
\end{picture} & 2 n - 1 & (n,n) \\
B_n & \hspace{5mm} \OG(2,{2n+1}) \hspace{5mm} &
\setlength{\unitlength}{2.5mm}
\begin{picture}(14,3)(-2,0)
\put(0,0){$\circ$}
\multiput(2,0)(2,0){5}{$\circ$}
\multiput(.73,.4)(2,0){4}{\line(1,0){1.34}}
\multiput(8.73,.2)(0,.4){2}{\line(1,0){1.34}}
\put(2,0){$\bullet$}
\end{picture} & 4 n - 5 & 2 n - 2 \\
C_n & \hspace{5mm} \p^{2n-1} \hspace{5mm} &
\setlength{\unitlength}{2.5mm}
\begin{picture}(14,3)(-2,0)
\put(0,0){$\circ$}
\multiput(2,0)(2,0){5}{$\circ$}
\multiput(.73,.4)(2,0){4}{\line(1,0){1.34}}
\multiput(8.73,.2)(0,.4){2}{\line(1,0){1.34}}
\put(0,0){$\bullet$}
\end{picture} & 2 n - 1 & 2 n  \\
D_n & \hspace{5mm} \OG(2,2n)\hspace{5mm} &
\setlength{\unitlength}{2.5mm}
\begin{picture}(14,3)(-2,0)
\put(2,0){$\circ$}
\multiput(0,0)(2,0){5}{$\circ$}
\multiput(0.73,.4)(2,0){4}{\line(1,0){1.34}}
\put(10,1.2){$\circ$}
\put(2,0){$\bullet$}
\put(10,-1.1){$\circ$}
\put(8.6,0.2){\line(5,-3){1.5}}
\put(8.6,0.5){\line(5,3){1.5}}
\end{picture}
 & 4 n - 7 & 2 n - 3 \\
E_6 &\hspace{5mm} E_6/P_2 \hspace{5mm}& 
\setlength{\unitlength}{2.5mm}
\begin{picture}(14,3)(-2,0)
\put(0,0){$\circ$}
\multiput(2,0)(2,0){4}{$\circ$}
\multiput(.73,.4)(2,0){4}{\line(1,0){1.34}}
\put(0,0){$\circ$}
\put(4,-2){$\bullet$}
\put(4.42,-1.28){\line(0,1){1.36}}
\end{picture}
& 21 & 11 \\
E_7 & \hspace{5mm}E_7/P_1\hspace{5mm} & 
\setlength{\unitlength}{2.5mm}
\begin{picture}(14,3)(-2,0)
\put(0,0){$\circ$}
\multiput(2,0)(2,0){5}{$\circ$}
\multiput(.73,.4)(2,0){5}{\line(1,0){1.34}}
\put(0,0){$\bullet$}
\put(4,-2){$\circ$}
\put(4.42,-1.28){\line(0,1){1.36}}
\end{picture}
& 33 & 17 \\
E_8 & \hspace{5mm}E_8/P_8\hspace{5mm} & 
\setlength{\unitlength}{2.5mm}
\begin{picture}(14,3)(-2,0)
\put(0,0){$\circ$}
\multiput(2,0)(2,0){6}{$\circ$}
\multiput(.73,.4)(2,0){6}{\line(1,0){1.34}}
\put(12,0){$\bullet$}
\put(4,-2){$\circ$}
\put(4.42,-1.28){\line(0,1){1.36}}
\end{picture}
& 57 & 29 \\
F_4 & \hspace{5mm} F_4/P_1 \hspace{5mm} &
\setlength{\unitlength}{2.5mm}
\begin{picture}(14,3)(-2,0)
\put(0,0){$\circ$}
\multiput(2,0)(2,0){3}{$\circ$}
\multiput(.73,.4)(2,0){1}{\line(1,0){1.34}}
\multiput(4.73,.4)(2,0){1}{\line(1,0){1.34}}
\multiput(2.73,.2)(0,.4){2}{\line(1,0){1.34}}
\put(0,0){$\bullet$}
\end{picture} & 15 & 8 \\
G_2 & \hspace{5mm} G_2/P_1 \hspace{5mm} &
\setlength{\unitlength}{2.5mm}
\begin{picture}(14,3)(-2,0)
\put(0,0){$\circ$}
\multiput(2,0)(2,0){1}{$\circ$}
\multiput(.73,.4)(2,0){1}{\line(1,0){1.34}}
\multiput(0.73,.2)(0,.4){2}{\line(1,0){1.34}}
\put(0,0){$\bullet$}
\end{picture} & 5 & 3 \\
\end{array}
$$

\medskip
\medskip

\centerline{Table 3. List of adjoint varietes.}

\section{Semisimplicity of the quantum cohomology}

In this section we apply the results in Section \ref{section-rad} to the varieties of Section \ref{section-QX} and get results on the semisimplicity of their quantum cohomology. 

\subsection{Linear sections of cominuscules homogeneous spaces}

We consider the varieties $Y$ obtained as linear sections of a cominuscule homogeneous space $X$ satisfying the assumptions of Theorem \ref{thm-def-pos}. These varietes are listed in Remark \ref{rem-def-pos}. 

\subsubsection{Multiplication with degree $2$ classes} We want to understand the endomorphism $E_h^Y$ of $\QH(Y)$ obtained by multiplication with $h$ the hyperplane class. Let $j : Y \to X$ be the inclusion and let $\tau \in \HH(Y)$. We denote by $h$ the hyperplane class in $\HH(X)$ and $\HH(Y)$ as well. Projection formula gives the following result.

\begin{lemma}
\label{lemm-class}
We have $j_*(h \cup \tau) = h \cup j_*\tau$.
\end{lemma}

As in the proof of Theorem \ref{thm-def-pos} (see also Remark \ref{rem-thm-def-pos}) we prove a result comparing Gromov-Witten invariants on $X$ and $Y$. Write $I_d(-,-,-)_X$ and $I_d(-,-,-)_Y$ for Gromov-Witten invariants of degree $d$ in $X$ and $Y$. 

\begin{lemma}
\label{lemm-pos-gen}
Let $\tau,\tau' \in \HH(Y)$ be cohomology classes such that the following conditions hold.
\begin{itemize}
\item $\deg \tau + \deg \tau' = 2 \dim Y + 2 c_1(Y) -2$.
\item There exists varieties $S,S'$ in $Y$ with $j_*\tau = [S]$, $j_*\tau' = [S']$ which are in general position in $X$.
\end{itemize}
Then $I_1(\tau,\tau',h)_Y = I_1(j_*\tau,j_*\tau',h)_X$
\end{lemma}

\begin{proof}
Note that the equality on degrees is equivalent to $\codim_X S +
\codim_X S' = \dim X + c_1(X) -1$ and $\codim_Y S + \codim_Y S' = \dim
Y + c_1(Y) -1$. This together with the second condition imply the following property: the scheme of $2$-points degree $1$ stable maps to $X$ passing through $S$ and $S'$ is finite and reduced. Remark that any degree $1$ stable map to $Y$ passing through $S$ and $S'$ is a degree $1$ stable map to $X$ passing through $S$ and $S'$ and conversely since $Y$ is a linear section of $X$. In particular the scheme of degree $1$ stable maps to $Y$ passing through $S$ and $S'$ is finite and reduced. This implies that the moduli space of degree $1$ stable maps to $Y$ has the expected dimension and that the above number of stable maps is equal to both $I_1(\tau,\tau',h)_Y$ and $I_1(j_*\tau,j_*\tau',h)_X$.
\end{proof}

\begin{prop}
Let $a,b$ be integers in $[0,\dim Y]$ such that 
\begin{itemize}
\item $a + b = \dim Y + c_1(Y) - 1$.
\item $k \geq c_1(Y)$.
\end{itemize}
Then there exists basis of classes $\tau,\tau'$ in $j^*H^{2a}(X,\R)$ and $j^*H^{2b}(X,\R)$ such that the assumptions of Lemma \ref{lemm-pos-gen} are satisfied.
\end{prop}

\begin{remark}
Note that $a,b \geq c_1(Y) -1$. By the Hard Lefschetz Theorem and since $2 c_1(Y) > \dim Y$ we get $j^*H^{2a}(X,\R) = H^{2a}(Y,\R)$ except if $a = c_1(Y) -1$ and $\dim Y = 2 c_1(Y) -1$. 
\end{remark}

\begin{proof}
We may assume $a \geq b$. We shall construct general hyperplane sections $Y = X \cap L_k$ with $L_k$ a general linear subspace of codimension $k$ satisfying the proposition.

\textbf{Case 1.} We first deal with the case where $a = \dim Y \textrm{ or } \dim Y - 1$ \emph{i.e.} there is a unique class $\tau$ in $H^{2a}(Y,\R)$: the class $\pt$ or a point or the class $\ell$ of a line.

We prove that $\tau$ and the pull-back via $j$ of the Schubert basis in $H^{2b}(X,\Z)$ satisfy the proposition. Let $S$ be a Schubert variety of codimension $b$ in $X$ and $S'$ be a point or a line (depending on whether $a = \dim Y$ or $a = \dim Y -1$). Let $\mathcal{S}$ be family of $G$-translates of $S$ and $\mathcal{T}$ the family of $G$-translates of $T$. Let $\mathcal{L}_k$ be the variety parametrising linear subspaces of codimension $k$ in the Pl\"ucker embedding of $X$. We have a rational morphism $f:G \times \mathcal{L}_k \to \mathcal{H}$ where $\mathcal{H}$ is a Hilbert scheme of subvarieties in $X$ defined by $(g,L_k) \mapsto gS \cap L_k$. Let $\mathcal{V}$ be the closure of its image.

We consider $I = \{(V,g'T,L_k) \in \mathcal{V} \times \mathcal{T} \times \mathcal{L}_k \ | \ V \subset  L_k \supset g'T\}$. We have a diagram
$$\xymatrix{I \ar[r]^-p \ar[d]_-q & \mathcal{V} \times \mathcal{T} \\
\mathcal{L}_k. & \\}$$
Since $b = \dim Y + c_1(Y) -1 -a$, the Schubert variety $S$ is contained in a linear section of $X$ of codimension $\dim Y + c_1(Y) -1 -a$ and thus $V = gS \cap L_k$ is contained in a linear section of $X$ of codimension $\dim Y + c_1(Y) -1 - a + k$. In particular since the space of $T$ has dimension $\dim Y - a$, the variety $V \cup g'T$ is contained in a linear section of codimension $c_1(Y) -1 + k \geq k$. This proves that the map $p$ is surjective. The map $q$ is also surjective: for $L_k \in \mathcal{L}_k$ pick for $g'T$ a point or a line in $X \cap L_k$, pick $g \in G$ general and set $V = g S \cap L_k$. 

In particular, for $L_k \in \mathcal{L}_k$ general, there is a translate $gS$ and a point or a line $T \subset X \cap L_k$ such that $V = gS \cap L_k$ and $T$ are in general position in $X$. Setting $Y = X \cap L_k$ proves the result.

\textbf{Case 2.} We now consider the other cases. We prove that the pull-backs via $j$ of the Schubert basis in $H^{2a}(X,\R)$ and $H^{2b}(X,\R)$ satisfy the proposition. Let $S$ and $S'$ be Schubert varieties in $X$ of respective codimension $a$ and $b$. Denote by $\mathcal{S}$ and $\mathcal{S}'$ the family of $G$-translates of $S$ and $S'$. These are projective homogeneous spaces since the stabiliser of $S$ and $S'$ contain a Borel subgroup. Let $\mathcal{L}_k$ be the projective variety of all linear subspaces of codimension $k$ in the Pl\"ucker embedding of $X$. 

We have a rational morphism $f:\mathcal{S} \times \mathcal{S}' \times \mathcal{L}_k \times  \mathcal{L}_k \to \mathcal{H} \times \mathcal{H}'$, where $\mathcal{H}$ and $\mathcal{H}'$ are Hilbert schemes of subvarieties in $X$, defined by $(gS,g'S',L_k,L_k') \mapsto (gS \cap L_k, g'S' \cap L_k')$. Let $\mathcal{V} \times \mathcal{V'}$ be the closure of the image of $f$.

\begin{lemma}
Let $\Delta : \mathcal{S} \times \mathcal{S}' \times \mathcal{L}_k \to \mathcal{S} \times \mathcal{S}' \times \mathcal{L}_k \times \mathcal{L}_k$ be the map induced by the diagonal embedding of $\mathcal{L}_k$. Then $\mathcal{V} \times \mathcal{V'}$ is the image of $f \circ \Delta$.
\end{lemma}

Assume that the lemma holds. Let $O$ be the open subset in $\mathcal{V} \times \mathcal{V'}$ of subvarieties in general position and let $\mathcal{L}_k^\circ$ be the open non empty subset in $\mathcal{L}_k$ of linear subspaces having a smooth intersection with $X$. Then $(f \circ \Delta)^{-1}(O)$ is non empty and open in $\mathcal{S} \times \mathcal{S}' \times \mathcal{L}_k$ as well as ${\rm pr}_3^{-1}(\mathcal{L}_k^\circ)$ where ${\rm pr}_3$ is the projection on the third factor in $\mathcal{S} \times \mathcal{S}' \times \mathcal{L}_k$. Since $\mathcal{S} \times \mathcal{S}' \times \mathcal{L}_k$ is irreducible, these subsets intersect. Let $(gS,g'S',L_k)$ be in the intersection. Then $Y = X \cap L_k$ and $V = Y \cap gS$, $V' = Y \cap g'S'$ satisfy the desired property since $V$ and $V'$ are in general position in $X$.

We are left to proving the lemma. Equivalently we need to prove that for $(gS,L_K) \in \mathcal{S} \times \mathcal{L}_k$ and $(g'S',L_k') \in \mathcal{S}' \times \mathcal{L}_k$ such that $\codim_{gS} gS \cap L_k = k = \codim_{g'S'} g'S' \cap L'_k$, there exists $L_k'' \in \mathcal{L}_k$ such that $gS \cap L_k \subset L_k'' \supset g'S' \cap L'_k $. Let $W$ be the vector space defining the Pl\"ucker embedding and let $\scal{gS}$ and $\scal{g'S'}$ be the spans in $W$ of element whose classes are in $gS$ and $g'S'$. It is enough to prove that $\dim \scal{gS} + \dim \scal{g'S'} \leq \dim W + k$. We prove this inequality by case by case analysis. 

For $X = \LG(3,6)$, we have $k = 1$ and $\dim Y + c_1(Y) -1 = 7$ thus $a \geq \dim Y -1$ and there is nothing to prove.

Note that in the other cases $X$ is a minuscule homogeneous space. This means that the weights of $W$ for a maximal torus form a unique orbit for the action of the Weyl group of $G$. This in particular implies that there is a correspondence between Schubert varieties and weights of $W$. As a consequence we get the equality
$$\dim \scal{S} = |\{S'' \subset S \ | \ \textrm{$S''$ a Schubert variety}\}|.$$
In words: the dimension of the span $\scal{S}$ of a Schubert variety $S$ is equal to the number of Schubert varieties contained in $S$. This translates the inequality $\dim \scal{gS} + \dim \scal{g'S'} \leq \dim W + k$ into a combinatorial computation and an easy case by case check gives the result.
\end{proof}

We shall now define maps between subspaces of $\QH(Y)$ and $\QH(X)$. Recall that we have morphisms $j^* : H^m(X,\Z) \to H^m(Y,\Z)$ and $j_* : H_m(X,\Z) \to H_m(X,\Z)$ which become isomorphisms for $m < \dim Y$ by Lefschetz Theorem. Let $A(Y)$ and $A(X)$ be the algebras obtained from $\QH(Y)$ and $\QH(X)$ by quotienting with the ideal $(q - 1)$. Recall that these algebras are respectively $\Z/2 c_1(Y)\Z$ and $\Z/2 c_1(X)\Z$ graded. We write $A_a(Y)$ and $A_a(X)$ for their degree $a$ graded piece. We have
$$\begin{array}{rll}
A_a(X) = & H^a(X,\R) \oplus H^{2 c_1(X) + a}(X,\R) & \textrm{for $a \in [0, 2 \dim Y - 2 c_1(Y)]$} \\
A_a(Y) = & H^a(Y,\R) \oplus H^{2 c_1(Y) + a}(Y,\R) & \textrm{for $a \in [0, 2 \dim Y - 2 c_1(Y)]$} \\
A_{2 c_1(Y) -2}(X) = & H^{2 c_1(Y) - 2}(X,\R) & 
\textrm{for $2 c_1(Y) - 1 > \dim Y$} \\
A_{2c_1(Y) - 2}(Y) = & H^{2 c_1(Y) - 2}(Y,\R) & 
\textrm{for $2 c_1(Y) - 1 > \dim Y$} \\
A_{2 c_1(Y) -2}(X) = & H^{2 c_1(Y) - 2}(X,\R) \oplus H^{2 \dim (X)}(X,\R) & 
\textrm{for $2 c_1(Y) - 1 = \dim Y$} \\
A_{2c_1(Y) - 2}(Y) = & H^{2 c_1(Y) - 2}(Y,\R) \oplus H^{2 \dim (Y)}(Y,\R) & 
\textrm{for $2 c_1(Y) - 1 = \dim Y$} \\
\end{array}$$

We define a morphism $J$ between these spaces as follows:
$$\begin{array}{rll}
J = j^* \oplus j_*^{-1} : & A_a(X) \to A_a(Y) & \textrm{for $a \in [2 c_1(Y), 2 \dim Y]$} \\
J = {j^*} : & A_{2c_1(Y)-2}(X) \to A_{2c_1(Y)-2}(Y) & \textrm{for $2 c_1(Y) - 1 > \dim Y$} \\
J = {j^* \oplus j_*^{-1}} : & A_{2 c_1(Y) -2}(X) \to A_{2c_1(Y) - 2}(Y) & \textrm{for $2 c_1(Y) - 1 = \dim Y$} \\
\end{array}
$$
Note that by Lefschetz Theorem and because $2 c_1(Y) > \dim Y$, the first map is an isomorphism for all $a \in [2 c_1(Y), 2 \dim Y]$.

\begin{cor}
\label{coro-rest}
For all $a \in [2 c_1(Y), 2 \dim Y - 2]$, we have a commutative diagrams
$$\xymatrix{
A_a(X) \ar[r]^-J_-\sim \ar[d]_-{E_h^X} & A_a(Y) \ar[d]^-{E_h^Y} \\
A_{a+2}(X) \ar[r]^-J_-\sim & A_{a+2}(Y)} \ \ \textrm{ and } \ \ \xymatrix{
A_{2c_1(Y)-2}(X) \ar[r]^-J \ar[d]_-{E_{h^{k+1}}^X} & A_{2c_1(Y)-2}(Y) \ar[d]^-{E_{h}^Y} \\
A_0(X) \ar[r]^-J_-\sim & A_0(Y).}$$
\end{cor}

\begin{proof}
For a quantum product $a \starz b$ we write $a \starz b = \sum_d q^d (a \starz b)_d$.

We start with the first square. Let $\sigma \in A_a(Y)$. We have $\deg \s = a$ or $\deg \s = 2c_1(X) + a$. In the first case $J\s = j^*\s$ and $\deg J\s + 2 = \deg \s + 2 < 2 c_1(Y) < 2 c_1(X)$. In particular $E_X(\s) = h \cup \s$ and $E_h^Y(Js) = h \cup J\s$ and we get $E_h^YJ(\s) = h \cup j^*\s = j^*(h \cup \s) = JE_h^X(\s)$. In the second case we have $\sigma = j_* \tau$ with $\tau = J\sigma$. We get $JE_h^X(\s) = J(h \starz j_*\tau)$ and $E_h^YJ(\s) = E_h^Y(\tau) = h \starz \tau$. By Lemma \ref{lemm-class} we have $j_*(h \starz \tau)_0 = (h \starz j_*\tau)_0$ so the result is true for the classical part of the quantum product. The quantum parts of $E_h^YJ(\s)$ and $JE_h^XJ(\s)$ are of the form 
$$(h \starz \tau)_1 = I_1(h,\tau,\ell)_Y q h \textrm{ and } J((h \starz j_*\tau)_1) = I_1(h,j_*\tau,\ell)_X qh.$$ 
Since $2 c_1(Y) > \dim Y$, the Hard Lefschetz Theorem implies that there is a $\sigma''$ with $\tau = j^*\sigma''$. Applying the above proposition, we get $I_1(h,\tau,\ell)_Y = I_1(h,j_*\tau,\ell)_X$ proving the result.

We now consider the second square. The possible degrees for $\s \in A_{2 c_1(Y) - 2}(X)$ are $2 c_1(Y) - 2$ or $2 \dim X$ if $\dim Y = 2c_1(Y) - 1$. First assume $\deg \s = 2 c_1(Y) -2$ and let $\tau = J(\s) = j^* \sigma$. We have $JE_{h^{k+1}}^X(\s) = J(h^{k+1} \starz \sigma)$ and $E_h^YJ(\s) = h \starz \tau$. For degree reasons, we have $h^k = [Y]$ and $[Y] \starz \sigma = [Y] \cup \sigma$. We get $E_{h^{k+1}}^X(\sigma) = h^{k+1} \starz \sigma = h \starz ([Y] \cup \sigma) = h \starz j_*j^* \sigma$. By Lemma \ref{lemm-class}, we have $j_*(h \starz \tau)_0 = j_*(h \cup \tau) = h \cup j_*\tau = h \cup j_*j^*\sigma$ and the result is true for the classical part of the quantum product. The quantum part of $E_h^YJ(\s)$ and $JE_{h^{k+1}}^X(\s)$ are of the form 
$$(h \starz \tau)_1 = I_1(h,\tau,\pt)_Y q \textrm{ and } J((h \starz j_*\tau)_1) = I_1(h,j_*\tau,\pt)_X q.$$ 
By the above proposition, we get $I_1(h,\tau,\pt)_Y = I_1(h,j_*\tau,\pt)_X$ proving the result for $\deg \s = 2 c_1(Y) -2$.

Finally assume $\dim Y = 2 c_1(Y) -1$ and $\deg \s = 2 \dim X$. We have $\s = \pt$ and $J(\s) = \pt$. We get $JE_{h^{k+1}}^X(\s) = J(h^{k+1} \starz \pt)$ and $E_h^YJ(\s) = h \starz \pt$. We have $h \starz \pt = q \sum_\gamma I_1(h,\pt,\gamma)_Y {\rm PD}(\gamma) + I_2(h,\pt,\pt)_Y q^2$
where the sum runs over a basis of $H^{2 c_1(Y) - 2}(Y,\R)$. By Lefschetz Theorem the classes $\delta$ with $J \delta = j^*\delta = \gamma$ form a basis of $H^{2 c_1(Y) - 2}(X,\R)$. By the above proposition, we have $I_1(h,\pt,\gamma)_Y = I_1(h,\pt,j_*\gamma)_X$. On the other hand, applying Remark \ref{rem-thm-def-pos} we have $I_2(h,\pt,\pt)_Y = I_2(j_*h,\pt,\pt)_X$. We get
$$\begin{array}{rl}
E_h^YJ(\s) = & q \sum_\sigma I_1(h,\pt,j_*j^*\delta)_X {\rm PD}(j^*\delta) + I_2(j_*h,\pt,\pt)_X q^2 \\
= & q \sum_\sigma I_1(h,\pt,[Y] \cup \delta)_X j^*{\rm PD}(\delta) + I_2([Y] \cup h,\pt,\pt)_X q^2 \\
 = & q \sum_\sigma I_1(h,\pt,[Y] \starz \delta)_X j^*{\rm PD}(\delta) + I_2([Y] \starz h,\pt,\pt)_X q^2 \\ 
 = & q \sum_\sigma I_1([Y] \starz h,\pt,\delta)_X j^*{\rm PD}(\delta) + I_2(h^{k+1},\pt,\pt)_X q^2\\
 = & q j^*\sum_\sigma I_1(h^{k+1},\pt,\delta)_X {\rm PD}(\delta) + I_2(h^{k+1},\pt,\pt)_X q^2\\
 = & JE_{h^{k+1}}^X(\pt) \end{array}
$$
This completes the proof.
\end{proof}

\begin{remark}
For $\dim Y < 2 c_1(Y) -2$ using the same arguments as above we get a slightly better factorisation. We have an isomorphism $J$ defined by
$$J = j_*^{-1} : A_{-2}(X) = H^{2 c_1(X) - 2}(X,\R) \to A_{-2}(Y) = H^{2 c_1(Y) - 2}(Y,\R).$$ 
For $\dim Y < 2 c_1(Y) - 2$, we have a commutative diagram
$$\xymatrix{A_{-2}(X) \ar[r]^-J_-\sim \ar[d]_-{E_h^X} & A_{-2}(Y) \ar[d]^-{E_h^Y} \\
A_0(X) \ar[r]^-J_-\sim & A_0(Y).}$$
\end{remark}

\subsection{Semisimple small quantum cohomology}

In this subsection, we prove the following semisimplicity result.

\begin{thm}
\label{thm-semisimple-Y}
Let $Y$ be a general hyperplane section with $2 c_1(Y) > \dim Y$ of the following homogeneous space $X$:
$$\begin{array}{llll}
\Gr(2,2n+1) & & & F_4/P_1 \\
\OG(5,10) & & & \OG(2,2n+1) \\
\LG(3,6) & & & G_2/P_1.\\
\end{array}$$
Then $\QH(Y)$ is semisimple.
\end{thm}

\begin{proof}
Note that for the varieties of the second column, we have $2 c_1(X) -1 = \dim X$ thus we must have $Y = X$. By Theorem \ref{thm-def-pos} and Proposition \ref{prop-def-pos}, the quadratic form $Q_Y$ is positive definite. By Theorem \ref{thm-alg-qh}, it is enough to prove that $\Ker E_h^Y = 0$ \emph{i.e.} that $E_h^Y$ is bijective. It is therefore enough to prove that $h$ is invertible \emph{i.e.} that $E_h^Y$ is surjective onto $A_0(Y)$. By Corollary \ref {coro-rest} it is enough to prove that $E_h^X$ is surjective onto $A_0(X)$. This is now an easy check using \cite{cmp1} and \cite{adjoint}.
\end{proof}

\begin{remark}
Note that this result together with the result of \cite{HMT} implies that the cohomology of the varieties $Y$ above is even and of pure type $(p,p)$.
\end{remark}

\subsection{Structure of the radical}

In this subsection, we describe the radical of $\QH(Y)$ for some Fano varieties $Y$.

\begin{prop}
\label{prop-loc-rad}
Let $Y$ be a general hyperplane section with $2 c_1(Y) > \dim Y$ of the following homogeneous space $X$:
$$\begin{array}{llll}
\Gr(2,2n) & & & \OG(2,2n) \\
\OG(6,12) & & & E_6/P_2 \\
E_6/P_6 & & & E_7/P_1\\
E_7/P_7 & & & E_8/P_8.\\
\end{array}$$
Then 
$$R(Y) = \Ker E_h^Y \cap \bigoplus_{2 c_1(Y) \not\ \!| \ k}\QH^{k}(Y).$$
\end{prop}

\begin{proof}
Note that for the varieties of the second column, we have $2 c_1(X) -1 = \dim X$ thus we must have $Y = X$. 

We first prove $R(Y) \subset \Ker E_h^Y$. By Theorem \ref{thm-def-pos} and Proposition \ref{prop-def-pos}, the quadratic form $Q_Y$ is positive definite. By Theorem \ref{thm-alg-qh}, it is enough to prove that $A_0(Y) = \Ker E_h^Y \cap A_0(Y) \oplus \im E_h^Y \cap A_0(Y)$. By Corollary \ref{coro-rest} this is equivalent to the same statement in $A_0(X)$. This is now an easy check using \cite{cmp1} and \cite{adjoint}. 

Note that in $\QH_0(Y)$, there is a unique vector $K \in \Ker E_h^Y$
and this vector is of the form $K = \lambda + v$ with $\lambda \in
\R\setminus \{0\}$ and $v \ni \im E_h^Y$. We thus have $K^n =
\lambda^{n-1} K$ and $K$ is not nilpotent. We now prove that any
element $K \in \Ker E_h^Y \cap \bigoplus_{2 c_1(Y) \not\ \!|
  \ k}\QH^{k}(Y)$ is nilpotent. Remark that for degree reasons there
will be a power $K^n$ of $K$ in $\im E_h^Y$ (by Corollary
\ref{coro-rest} and case by case inspection on $X$. We refer to
Subsection \ref{sec-ig2n} for a detailled proof of this in the case $X
= \Gr(2,2n)$). But $K^n$ is in $\Ker E_h^Y$ thus $K^n = 0$.  
\end{proof}

\begin{example}
Let $X = \Gr(2,6)$ and $Y$ be a general hyperplane section. Then $Y$ is
isomorphic to an isotropic Gra{\ss}mann variety $\IG(2,6)$. We have
$\dim Y = 7$ and $c_1(Y) = 5$. The dimensions of the graded parts of
$A(Y)$ are 
$$\begin{array}{l|ccccc}
k & 0 & 2 & 4 & 6 & 8 \\
\hline
A_k(Y) & 3 & 2 & 3 & 2 & 2 \\
\end{array}$$
One easily checks that $\Ker E_h^Y$ has dimension 2. There is an
element $K_0$ of degree $0$ in $\Ker E_h^Y$ and an element $K_4$ of
degree $4$ in $\Ker E_h^Y$. The image $\im E_h^Y$ is a complement of
$\Ker E_h^Y$. As in the above proposition $K_0$ is not nilpotent and
$K_4^2$ has degree $8$ this is in $\im E_h^Y$ so $K_4^2 = 0$ and we
have $R(Y) = \scal{K_4}$. We recover the example in \cite[Section
  7]{semisimple}.
\end{example}

\section{Big quantum cohomology}

In this section we consider the case of two Fano varieties $Y$ obtained as
hyperplane section of a cominuscule homogeneous space $X$ whose small
quantum cohomology $\QH(Y)$ is not semisimple. We however prove that
the big quantum cohomology $\BQH(Y)$ is semisimple. These are the
first examples of semisimplicity of the big quantum cohomology in the
presence of non semisimple small quantum cohomology. 

The varieties we consider are $Y = \IG(2,2n)$ obtained as hyperplane section of $X = \Gr(2,2n)$ and $Y = F_4/P_4$ obtained as hyperplane section of $X = E_6/P_6$. Note that both varieties are homogeneous and actually \emph{coadjoint varieties} in the sense of \cite{adjoint}. Their small quantum cohomology is not semisimple but well understood. We refer to \cite{adjoint} for more details. In the next subsection, we recall few fact on $\QH(Y)$ for $Y$ coadjoint.

\subsection{Quantum cohomology of coadjoint varieties}

Let $Y$ be one of the following two varieties $Y = \IG(2,2n)$ or $Y = F_4/P_4$ which are homogeneous under the action of a reductive group $G$ of type $C_n$ of $F_4$ respectively.

\subsubsection{Cohomology and short roots}

The cohomology of a coadjoint variety $Y$ homogeneous under the action of a reductive group $G$ is easily described using Schubert classes. There are several indexing sets for Schubert varieties. We will choose the indexing set described in \cite{adjoint}. Let $R_s$ be the set of short roots in $R$ the root system of $G$. For each root $\a \in R_s$, there is a cohomology class $\s_\a$ and the family $(\s_\a)_{\a \in R_s}$ form a basis of $\HH(X)$. 

Let $n$ be the rank of the group (and of the roots system $R$). We choose $(\a_1 , \cdots , \a_n)$ a Basis of the root system with notation as in \cite{bou}. For a root $\a$ we have an expression 
$$\a = \sum_{i = 1}^n a_i \a_i$$ 
with $a_i \in \Z$ for all $i \in [1,n]$. We define the height $\htt(\a)$ of $\a \in R$ by 
$$\htt(\a) = \sum_{i = 1}^n a_i.$$
Let $\theta$ be the highest short root of $R$. The above indexing satisfies many nice properties. We have (see \cite[Proposition 2.9]{adjoint})
$$\deg(\s_\a) = \left\{ \begin{array}{ll}
 2(\htt(\theta) - \htt(\a)) & \textrm{ for $\a$ positive,} \\
 2(\htt(\theta) - \htt(\a) -1) & \textrm{ for $\a$ negative.} \\ 
 \end{array}\right.$$
We will write $1$ for the class $\s_{\theta}$ and $h$ for the hyperplane class in the Pl\"ucker embbeding of $X$. The Poincar\'e duality has a very simple form on roots: the Poincar\'e dual $\s_\a^\vee$ of $\s_\a$ is simply $\s_\a^\vee = \s_{-\a}$ (see \cite[Proposition 2.9]{adjoint}).

\subsubsection{Small quantum cohomology and affine short roots}

The above parametrisation of Schubert classes by short roots can be
extended to quantum monomials in $\QHl = \QH(X)[q,q^{-1}]$. A quantum monomial is a element $q^d \s_\a$ where $d \in \Z$ and $\a \in R_s$. We write $\QM(X)$ for the set of all quantum monomials and we write $\QHlh{k}$ for the degree $k$ part of $\QHl$.

Let $\Rh$ be the extended affine root system of $R$ and let $\delta$
be the minimal positive imaginary root. The extended root system has basis
$(\a_0, \cdots , \a_n)$ and in this basis we have $\delta = \Theta +
\a_0$ where $\Theta$ is the highest root of $R$. A short root of $\Rh$ is a root
of the form $\a + d \delta$ for $\a \in R_s$ and $d \in \Z$. We write
$\Rh_s$ for the set of short roots in $\Rh$. There is a bijection 
$$\eta : \Rh_s \to \QM(X)$$
defined by $\eta(\a - d \delta) = q^d \s_\a$. Note that we can extend the height function on $\Rh$ and that we have $\deg(q) = 2(\htt(\delta) - 1)$. 

\subsubsection{Multiplication with the hyperplane}
We have the following very simple description of the small quantum product $\star_0$ with $h$ (see \cite[Theorem 3]{adjoint}):
$$h \star_0 \s_\a = \sum_{i \in [0,n],\ \scal{\a_i^\vee,\a} > 0}
\scal{\a_i^\vee,\a} \eta(s_{\a_i}(\a)).$$

\subsection{Gra{\ss}mann variety of lines}
\label{sec-ig2n}

Let $X = \Gr(2,2n)$ and let $Y$ be a linear section of $X$ of
codimension $1$. 
Note that $Y = \IG(2,2n)$ is an isotropic Gra{\ss}mann
variety. The small quantum cohomology of this variety is described in \cite{adjoint}. We first prove the following result which improves Proposition \ref{prop-loc-rad}. 

\begin{lemma}
Let $Y$ be a linear section of $X = \Gr(2,2n)$ of codimension $1$.

1. There exists a unique element $K_{4n+2}$ modulo scalar in $\QH^{4n+2}(Y) \cap
\Ker E_h^Y$.

2. The element $K_{4n+2}^k$ is divible by $q^k$. 

3. We have $R(Y) = \scal{K_{4n+2},\cdots,q^{-(n-2)}K_{4n+2}^{n-2}}$ and $K_{4n+2}^{n-1} = 0$. 
\end{lemma}

\begin{proof}
1. Note that all odd dimension cohomolgy groups of $X$ and $Y$
vanish. Recall also that we have the inclusion $R(Y) \subset \Ker
E_h^Y$. Using the description of Schubert classes with partitions
having 2 parts, we
have
$$\dim \QH^{2a}(X) = \left\{ \begin{array}{ll}
n & \textrm{for $a$ even} \\
n - 1 & \textrm{for $a$ odd.} \\
\end{array}
\right.$$
We deduce results on the dimension of the graded
parts in $\QH(Y)$. Recall that multiplication with $q$ induces an
isomorphism $\QH^{2a}(Y) \to \QH^{2a + 2 c_1(Y)}(Y)$ so we only need
to describe these dimensions for $0 \leq a < c_1(Y) = 2n - 1$.
$$\dim \QH^{2a}(Y) = \left\{ \begin{array}{ll}
n & \textrm{for $a \leq 2n - 4$ even} \\
n - 1 & \textrm{for $a = 2n - 2$} \\
n - 1 & \textrm{for $a \leq 2n -3$ odd.} \\
\end{array}
\right.$$
We work modulo the ideal $(q-1)$. This is enough since we can recover the powers of $q$ by considering the degrees. An easy check gives that $E_h^X : A_{2a}(X)
\to A_{2a + 2}(X)$ is of (maximal) rank $n-1$. Corollary
\ref{coro-rest} implies that the same holds for $E_h^Y$. In particular
$\Ker E_h^Y = \scal{K_{4n - 2} , K_{4n-2},\cdots,K_{8n - 10}}$ for some $K_{a} \in
A_{a}(Y) \cap \Ker E_h^Y$. Furthermore $\im E_h^Y$ is a complement of this space. This implies that $K_{4n-2} = \lambda + v$ with $v \in \im
E_h^Y$ and since $\Ker E_h^Y \cap \im E_h^Y = 0$, we have $K_{4n-2}^N =
\lambda^{N-1}K_{4n-2} \neq 0$ and $K_{4n-2}$ is not nilpotent. We claim that
modulo scalar we have $K_{4n+2}^i = K_{4i}$ for all $i \in [1,n-2]$. 
Let $\s_{(1,1)} \in H^{4}(X,\Z)$ be the Schubert class defined by the
partition $(1,1)$ (this Schubert class is also the top Chern class of the
tautological subbundle of $X$). It is easy to check that
$\s_{(1,1)}^{n-2} = \s_{(n-2,n-2)}$ this last class is the Schubert class associated to
the partition $(n-2,n-2)$. Indeed for degree reasons, this product is
a classical cohomological product and the result follows from
Littelwood-Richardson's rule. We get $h \cup \s_{(n-2,n-2)} \neq
0$. This implies that for $j:Y \to X$ the inclusion, we have
$j^*\s_{(1,1)}^{n-2} = j^*\s_{(n-2,n-2)}$ and that $j^*\s_{(1,1)}^{n-2}
\cup h \neq 0$. In particular $j^*\s_{(1,1)}^{n-2} \not \in \Ker
E_h^Y$ thus $j^*\s_{(1,1)} \not \in \Ker E_h^Y$ and we may write
$j^*\s_{(1,1)}^{n-2} = \mu K_{8n - 10} + w$ and $j^*\s_{(1,1)} = \lambda K_{4n+2} + v$ for some $\lambda,\mu \in \R \setminus
\{0 \}$ and $v,w \in \im E_h^Y$. This implies $\mu K_{8n - 10} + w =
j^*\s_{(1,1)}^{n-2} = \lambda^{n-2} K_{4n+2} + v^{n-2}$ and thus $\mu
K_{8n-10} = \lambda^{n-2} K_{4n+2}^{n-2}$ proving the claim. Now $K_{4n+2}^{n-1}
\in A_{4n-4}(Y) \subset \im E_h^Y$ (recall that we have a $\Z/2c_1(Y)\Z = \Z/(4n-2)\Z$-grading) thus $K_{4n+2}^{n-1} = 0$ proving the
result.
\end{proof}

For technical reasons we have to distinguish the cases $n = 3$ and $n \geq 4$ in the proof of the next result.

\begin{thm}
\label{thm-bqh1}
Let $Y = \IG(2,2n)$. Then $\BQH(Y)$ is semisimple.
\end{thm}

\begin{proof}
We will use the notation of \cite{bou} to index simple roots of the root system of type $\textrm{C}_n$. For classes $\s,\s' \in \HH(Y)$ we will write
$$\s \star_\tau \s' = \s \starz \s' + t (\s \star_1 \s') + O(t^2).$$

\textbf{We start with the case $n = 3$}. Let $\tau = \pt = \s_{-\theta}$ where $\theta$ is the highest short root. We prove that $\BQH_\tau(Y)$ (with notation as in Subsection \ref{sub-def}) is semisimple. 

Let $C \in \BQH_\tau(Y)$ be a nilpotent element of order $2$. If $C \neq 0$, up to dividing with the parameter $t$ (see Subsection \ref{sub-def}), we may assume that $C$ is of the form 
$$C = C_0 + t C_1 + O(t^2) \textrm{ with $C_0 \neq 0$}.$$
We have $0 = C \star_\tau C = C_0 \starz C_0 + O(t)$ thus $C_0 \in R(Y)$. Up to rescaling, we may assume that $C_0 = q \s_{\a_2 + \a_3} - q \s_{\a_2 + \a_3} + \s_{-\theta} \in \QH^{14}(Y)$. On the other hand, the product $h \star_\tau C$ has to be nilpotent. But we have 
$$h \star_\tau C = h \starz C_0 + t ( h \starz C_1 + h \star_1 C_0) + O(t^2) = t ( h \starz C_1 + h \star_1 C_0) + O(t^2).$$ 
In particular $h \starz C_1 + h \star_1 C_0$ is nilpotent. Since $h \star_1 C_0$ is of degree 28 and since there is no nilpotent element of degree 28 in $\QH(Y)$, we must have  $h \starz C_1 + h \star_1 C_0 = 0$. This is possible since $E_h^Y$ is surjective on degree 28. The element $C_1$ is thus uniquely determined by $h \starz C_1 = - h \star_1 C_0$. 
Note that $C_1 \in \QH^{26}(Y) \subset \im E_h^Y$ and $C_0 \in \Ker E_h^Y$ thus $C_0 \starz C_1 = 0$.

We now remark the following equality $C_0 = 3 \s_{-\theta} + (q \s_{\a_2 + \a_3} - q \s_{\a_2 + \a_3} -2 \s_{-\theta})$ where the second term is in the image of $E_h^Y$. Thus there exists $D_0 \in \QH^{12}(Y)$ with $h \starz D_0 = q \s_{\a_2 + \a_3} - q \s_{\a_2 + \a_3} -2 \s_{-\theta}$. An easy computation gives 
$$D_0 = q \s_{\a_1 + \a_2 + \a_3} - 2 \s_{-\a_1-\a_2-\a_3}.$$ 
We have $h \star_\tau D_0 = (q \s_{\a_2 + \a_3} - q \s_{\a_2 + \a_3} -2 \s_{-\theta}) + t (h \star_1 D_0) + O(t^2)$ with $h \star_1 D_0 \in \QH^{26}(Y)$. Since $E_h^Y$ is surjective on degree 26, there exists $D_1 \in \QH^{24}(Y)$ with $h \starz D_1 = - h \star_1 D_0$. Setting $D = D_0 + t D_1$ we get $h \star_\tau D = q \s_{\a_2 + \a_3} - q \s_{\a_2 + \a_3} -2 \s_{-\theta} + O(t^2)$. This altogether gives $C_0 = 3 \pt + h \star_\tau D + O(t^2)$ and 
$$3 \pt = C - h \star_\tau D - t C_1 + O(t^2).$$ 
Computing the square gives (recall that $C \star_\tau C = 0$, $h \star_\tau C = O(t^2)$ and $C_0 \starz C_1 = 0$)
$$\begin{array}{ll}
9 \pt \star_\tau \pt 
& = h \star_\tau D \star_\tau h \star_\tau D - 2 t C \star_\tau C_1 + 2 t h \star_\tau D \star_\tau C_1 + O(t^2) \\
& = h \star_\tau D \star_\tau h \star_\tau D - 2 t C_0 \starz C_1 + 2 t h \starz D \starz C_1 + O(t^2) \\
& = h \star_\tau D \star_\tau h \star_\tau D + 2 h \starz D \starz C_1 + O(t^2) \\
\end{array}$$
Note that we have
$$\begin{array}{rl}
h \star_\tau D \star_\tau h \star_\tau D = & h \starz h \starz D_0 \starz D_0 \\
 & + 2 t h \starz h \starz D_0 \starz D_1 + t h \starz h \starz (D_0 \star_1 D_0) \\
 & + t h \starz ( h \star_1 (D_0 \starz D_0)) + t h \star_1 (h \starz D_0 \starz D_0) + O(t^2)\\
 = & t h \star_1 (h \starz D_0 \starz D_0) + t \cdot \im E_h^Y + O(t^2).\\
\end{array}$$  
Finally we obtain $9 \pt \star_\tau \pt = 9 \pt \starz \pt + t h \star_1 (h \starz D_0 \starz D_0) + t \cdot \im E_h^Y + O(t^2)$. Since $D_0$ is explicitely given and since the endomorphism $h \star_1 -$ is understood using $\pt \starz -$ (see Subsection \ref{sub-def}) we get
$$9 \pt \star_1 \pt = 12 q^4 + 3 q^3 \s_{-\a_1-\a_2} - 3 q^3 \s_{-\a_2-\a_3} + \im E_h^Y = 6 q^4 + \im E_h^Y.$$

On the other hand we compute directly the product $\pt \star_\tau \pt$. We have
$$\pt \star_1 \pt = \sum_{d \geq 0} \sum_{\a \in R_s} q^d I_d(\pt,\pt,\pt,\s_{-\a}) \s_{\a}.$$ 
The invariant $I_d(\pt,\pt,\pt,\s_{-\a})$ vanishes unless $3 \deg \pt + \deg \s_{-\a} = 2 \dim Y + 2d c_1(Y) + 2$ thus $\deg \s_{-\a} = 10 d - 26$. Since $\deg \s_{-\a}$ is even and contained in $[0,14]$, the invariant $I_d(\pt,\pt,\pt,\s_{-\a})$ vanishes unless $d = 3$ and $\deg \s_{-\a} = 4$ or $d = 4$ and $\deg \s_{-\a} = 14$. But one easily check that there is no degree 3 curve in $Y = \IG(2,6)$ passing through 3 points in general position. Thus $I_3(\pt,\pt,\pt,-) = 0$. The only non vanishing invariant is thus $I_4(\pt,\pt,\pt,\pt)$ and we have 
$$9 \pt \star_1 \pt = 9 I_4(\pt,\pt,\pt,\pt) q^4 = 6 q^4 + \im E_h^Y.$$ 
Since $q^4 \not \in \im E_h^Y$ this implies $9 I_4(\pt,\pt,\pt,\pt) = 6$ which is impossible since $I_4(\pt,\pt,\pt,\pt) \in \Z_{\geq0}$.

\medskip

\textbf{Assume $n \geq 4$.}
Let $\tau = \s_{\theta - \a_1 - \a_2} \in \QH^{4}(Y)$ where $\theta$ is the highest short root. We prove that $\BQH_\tau(Y)$ (with notation as in Subsection \ref{sub-def}) is semisimple. 

Let $C \in \BQH_\tau(Y)$ be a nilpotent element of order $2$. If $C \neq 0$, up to dividing with the parameter $t$ (see Subsection \ref{sub-def}), we may assume that $C$ is of the form 
$$C = C_0 + t C_1 + O(t^2) \textrm{ with $C_0 \neq 0$}.$$
We have $0 = C \star_\tau C = C_0 \starz C_0 + O(t)$ thus $C_0 \in R(Y)$. Up to rescaling and multiplying with the generator of $R(Y)$, we may assume that $C_0$ is the unique element in $R(Y)$ of degree $8n-10$ such that the coefficient of $\s_{-\theta}$ in $C_0$ is 1. The product $h \star_\tau C$ has to be nilpotent. But we have 
$$h \star_\tau C = h \starz C_0 + t ( h \starz C_1 + h \star_1 C_0) + O(t^2) = t ( h \starz C_1 + h \star_1 C_0) + O(t^2).$$ 
In particular $h \starz C_1 + h \star_1 C_0$ is nilpotent. Since $h \star_1 C_0$ is of degree $8n - 6$ and since there is no nilpotent element of degree $8n - 6$ in $\QH(Y)$, we must have  $h \starz C_1 + h \star_1 C_0 = 0$. This is possible since $E_h^Y$ is surjective on degree $8n - 6$. The element $C_1$ is thus uniquely determined by $h \starz C_1 = - h \star_1 C_0$. 
Note that $C_1 \in \QH^{8n - 4}(Y) \subset \im E_h^Y$ and $C_0 \in \Ker E_h^Y$ thus $C_0 \starz C_1 = 0$ (one can check that $h \star_1 C_0 \neq 0$ but we do not need this since if $h \star_1 C_0 = 0$, then $h \starz C_1$ has to be nilpotent which implies $h \starz C_1 = 0$ thus $C_1 \in R(X)$ and $C_0 \starz C_1 = 0$).

We now remark the following equality $C_0 = n \s_{-\theta} + v$ with $v \in \im E_h^Y$. Thus there exists $D_0 \in \QH^{8n-12}(Y)$ with $h \starz D_0 = v$. 
We have $h \star_\tau D_0 = v + t (h \star_1 D_0) + O(t^2)$ with $h \star_1 D_0 \in \QH^{8n - 8}(Y)$. Since $E_h^Y$ is surjective on degree $8n - 8$, there exists $D_1 \in \QH^{8n - 10}(Y)$ with $h \starz D_1 = - h \star_1 D_0$. Setting $D = D_0 + t D_1$ we get $h \star_\tau D = v + O(t^2)$. This altogether gives $C_0 = n \ \pt + h \star_\tau D + O(t^2)$ and 
$$n \ \pt = C - h \star_\tau D - t C_1 + O(t^2).$$ 
Computing the square gives
$$n^2 \ \pt \star_\tau \pt = n^2 \ \pt \starz \pt + t h \star_1 (h \starz D_0 \starz D_0) + t \cdot \im E_h^Y + O(t^2).$$ 
Now $h \starz D_0 \starz D_0 = v \starz D_0 = (C_0 - n \s_{-\theta}) \starz D_0$ and since $D_0 \in \QH^{8n-12}(Y) \subset \im E_h^Y$ we have $C_0 \starz D_0 = 0$ thus $h \starz D_0 \starz D_0 = - n \ \pt \starz D_0$. Let $A = \pt \starz D_0$. To compute $A$ we first compute $h \starz A = \pt \starz h \starz D_0 = \pt \starz (C_0 - n \ \pt ) = - n \ \pt \starz \pt$ (since $\pt \in \scal{C_0,\im E_h^Y}$ and $C_0$ is orthogonal to both terms we have $C_0 \starz \pt = 0$). The Product $\pt \starz \pt = q^2 \s$ with $\s = \s_{\a_1 + \a_2 + \a_3 - \theta}$ is easy to compute using for example the kernel and span technique presented in \cite{BKT}. In particular we get $h \starz A = - n \ q^2 \s$. Thus for $\s' = \s_{\a_1 + 2 \a_2 + \a_3 - \theta}$ and $K$ the unique element in $\Ker E_h^Y$ of degree $16 n - 22 = \deg(A)$ with coefficient 1 on $q^2 \s'$, since $h \starz \s' = \s$, we have 
$$A = \lambda K - n \ q^2 \s'$$
for some scalar $\lambda$. Let $\ell$ be the class of a line in $Y$. Note that the coefficient of $\ell$ in $D_0$ is $1-n$. To compute $\lambda$, first note that $\pt \starz \textrm{PD}(\s') = q^2 h$ (again obtained using the kernel and span technique). In particular, the only non vanishing invariant of the form $I_d(\pt,\textrm{PD}(\s'),-)$ is the invariant $I_2(\pt,\textrm{PD}(\s'),\ell) = 1$. This implies that the coefficient of $q^2 \s'$ in $A = \pt \starz D_0$ is the coefficient of $\ell$ in $D_0$ and has value $1-n$. We deduce that $\lambda = 1$ and 
$$h \starz D_0 \starz D_0 = n^2 \ q^2 \s' - n \ K.$$
We now compute $h \star_1 (h \starz D_0 \starz D_0)$. We actually only need to compute modulo $n^2$ so that we only to consider $h \star_1 K$. An easy degree argument gives that there are only 2 Schubert classes appearing in $K$ with non vanishing value under $h \star_1 - $: the classes $q^3 \s_{a_1 + \a_2 + \a_3 + \a_4 - \theta}$ and $q^3 \s_{\a_2 + \a_3 + \a_4 + \a_5 - \theta}$ (only the first one for $n = 4$). One then easily check using the kernel and span technique the following formula
$$h \star_1 K = q^3 \s_{\a_2 + \a_3 + \a_4 + \a_5} \textrm{ (for $n = 4$ we have $h \star_1 K = \s_{\a_1 + \a_2 + \a_3 + \a_4}$).}$$
Write $q^3 \gamma = h \star_1 K$. The class $\gamma$ is a Schubert class and is not contained in $\im E_h^Y$. Altogether working in the Schubert basis modulo $n^2$ and modulo $\im E_h^Y$ we get:
$$- n q^2 \gamma \equiv n^2 \pt \star_1 \pt \equiv 0 \textrm{ (mod. $n^2$ and $\im E_h^Y$)}.$$
A contradiction.
\end{proof}

\subsection{Cayley plane}

In this subsection we consider the variety $Y = F_4/P_4$ obtained as hyperplane section of the Cayley plane $\O\p^2 = E_6/P_6$. Since the arguments are very similar to the case $Y = \IG(2,2n)$ and since the computer program \cite{programme} gives a complete description of the small quantum cohomology we shall only state the results and give a sketch of proof.

\begin{lemma}
Let $Y$ be a linear section of $X = E_6/P_6$ of codimension $1$. We have $\Ker E_h^Y = R(Y) = \scal{K_8}$ for some element $K_8 \in \QH^8(Y)$. 
\end{lemma}

\begin{proof}
Follows from the description of $E_h^Y$ using short roots and the fact that the element $K_8 \in \Ker E_h^Y \cap \QH^8(Y)$ satisfies $K_8^2 \in A_5(Y) \subset E_h^Y$ and $K_8^2 \in \Ker E_h^Y$ thus $K_8^2 = 0$ since $\im E_h^Y \cap \Ker E_h^Y = 0$.
\end{proof}

\begin{thm}
\label{thm-bqh2}
The big quantum cohomology $\BQH(Y)$ is semisimple.
\end{thm}

\begin{proof}
For classes $\s,\s' \in \HH(Y)$ we will write
$$\s \star_\tau \s' = \s \starz \s' + t (\s \star_1 \s') + O(t^2).$$
Let $(\a_1,\a_2,\a_3,\a_4)$ a system of simple roots of the root system of $F_4$ with $\a_3$ and $\a_4$ short. We have $\theta = \a_1 + 2 \a_2 + 3 \a_3 + 2 \a_4$ (recall that $\theta$ is the highest short root). 
Let $\tau = \pt$, we prove that $\BQH_\tau(Y)$ is semisimple. Let $C \in \BQH_\tau(Y)$ be nilpotent with $C \neq 0$ and $c \star_\tau C = 0$. We may assume that 
$$C = C_0 + t C_1 + O(t^2) \textrm{ with $C_0 \neq 0$}.$$
We have $0 = C \star_\tau C = C_0 \starz C_0 + O(t)$ thus $C_0 \in R(Y)$ and modulo rescaling we may assume $C_0 = \s_{-\theta} - q \s_{\a} + q \s_{\be} \in \QH^{30}(Y)$ with $\a = \a_1+\a_2+\a_3+\a_4$ and $\be = \a_2 + 2 \a_3 + \a_4$. We have:
$$h \star_\tau C = h \starz C_0 + t (h \star_1 C_0 + h \starz C_1) + O(t^2) = t (h \star_1 C_0 + h \starz C_1) + O(t^2).$$
In particular since $h \star_\tau C$ is nilpotent, we get that $h \star_1 C_0 + h \starz C_1$ is nilpotent. Since $h \star_1 C_0 = \in \QH^{60}(Y) = q^2 \QH^{16}(Y)$ and since there is no nilpotent element in degree 60 we get $h \star_1 C_0 + h \starz C_1 = 0$. This is possible since $E_h^Y$ is surjective on degree 60. The element $C_1$ is thus uniquely determined by $h \starz C_1 = - h \star_1 C_0$. 
Note that $C_1 \in \QH^{58}(Y) \subset \im E_h^Y$ and $C_0 \in \Ker E_h^Y$ thus $C_0 \starz C_1 = 0$.

We now remark the following equality $C_0 = 3 \s_{-\theta} + (q \s_{\be} - q \s_{\a} -2 \s_{-\theta})$ where the second term is in the image of $E_h^Y$. Thus there exists $D_0 \in \QH^{28}(Y)$ with $h \starz D_0 = q \s_{\be} - q \s_{\a} -2 \s_{-\theta}$. An easy computation gives 
$$D_0 = q \s_{\gamma} - 2 \s_{-\delta}$$
with $\gamma = \a_1 + \a_2 + 2 \a_3 + \a_4$ and $\delta = \a_1 + 2 \a_2 + 3 \a_3 + \a_4$. We have $h \star_\tau D_0 = (q \s_{\be} - q \s_{\a} -2 \s_{-\theta}) + t (h \star_1 D_0) + O(t^2)$ with $h \star_1 D_0 \in \QH^{58}(Y)$. Since $E_h^Y$ is surjective on degree 58, there exists $D_1 \in \QH^{56}(Y)$ with $h \starz D_1 = - h \star_1 D_0$. Setting $D = D_0 + t D_1$ we get $h \star_\tau D = q \s_{\be} - q \s_{\a} -2 \s_{-\theta} + O(t^2)$. This altogether gives $C_0 = 3 \pt + h \star_\tau D + O(t^2)$ and 
$$3 \pt = C - h \star_\tau D - t C_1 + O(t^2).$$ 
Computing the square gives as in the proof of Theorem \ref{thm-bqh1}
$$9 \pt \star_\tau \pt = \pt \starz \pt + t h \star_1 (h \starz D_0 \starz D_0) + t \cdot \im E_h^Y + O(t^2).$$ 
Since $D_0$ is explicitely given and since the endomorphism $h \star_1 -$ is understood using $\pt \starz -$ (see Subsection \ref{sub-def}) we get
$$9 \pt \star_1 \pt = 12 q^4 + 12 q^3 \s_{-\a} + 6 q^3 \s_{-\be} + \im E_h^Y = 6 q^4 + \im E_h^Y.$$

On the other hand we compute directly the product $\pt \star_\tau \pt$. We have
$$\pt \star_1 \pt = \sum_{d \geq 0} \sum_{\zeta \in R_s} q^d I_d(\pt,\pt,\pt,\s_{-\zeta}) \s_{\zeta}.$$ 
The invariant $I_d(\pt,\pt,\pt,\s_{-\zeta})$ vanishes unless $3 \deg \pt + \deg \s_{-\zeta} = 2 \dim Y + 2d c_1(Y) + 2$ thus $\deg \s_{-\zeta} = 22 d - 58$. Since $\deg \s_{-\zeta}$ is even and contained in $[0,30]$, the invariant $I_d(\pt,\pt,\pt,\s_{-\zeta})$ vanishes unless $d = 3$ and $\deg \s_{-\zeta} = 8$ or $d = 4$ and $\deg \s_{-\zeta} = 30$. But one it was proved in \cite[Proposition 2.13]{rational} that there is no degree 3 curve in $Y$ passing through 3 points in general position. Thus $I_3(\pt,\pt,\pt,-) = 0$. The only non vanishing invariant is thus $I_4(\pt,\pt,\pt,\pt)$ and we have 
$$9 \pt \star_1 \pt = 9 I_4(\pt,\pt,\pt,\pt) q^4 = 6 q^4 + \im E_h^Y.$$ 
Since $q^4 \not \in \im E_h^Y$ this implies $9 I_4(\pt,\pt,\pt,\pt) = 6$ which is impossible since $I_4(\pt,\pt,\pt,\pt) \in \Z_{\geq0}$.
\end{proof}

\addtocontents{toc}{\protect\setcounter{tocdepth}{2}}
\providecommand{\bysame}{\leavevmode\hbox to3em{\hrulefill}\thinspace}
\providecommand{\MR}{\relax\ifhmode\unskip\space\fi MR }
\providecommand{\MRhref}[2]{%
  \href{http://www.ams.org/mathscinet-getitem?mr=#1}{#2}
}
\providecommand{\href}[2]{#2}

\end{document}